\documentclass[12pt, a4paper]{amsart}
\usepackage[T1]{fontenc} 
\usepackage{graphicx} 
\usepackage{tikz}
\usepackage{float}
\usepackage{subcaption}
\usepackage{placeins}
\usepackage{amsmath, amssymb}
\usepackage{dsfont}
\usepackage{amsthm}
\usepackage[hidelinks]{hyperref}
\usepackage{cleveref}
\usepackage{a4wide}
\usepackage{xargs}
\usepackage{thmtools}
\usepackage{thm-restate}
\usepackage{pgfplots}
\pgfplotsset{compat=1.18}

\usepackage{todonotes} 

\newif\ifthesis
\thesisfalse

\newtheorem{theorem}{Theorem}[section]

\newtheorem{lemma}[theorem]{Lemma}
\newtheorem{definition}[theorem]{Definition}
\newtheorem{corollary}[theorem]{Corollary}

\newtheorem{proposition}[theorem]{Proposition}

\crefformat{theorem}{#2#1#3}
\crefformat{fact}{#2#1#3}
\crefformat{lemma}{#2#1#3}
\crefformat{corollary}{#2#1#3}
\crefformat{result}{#2#1#3}
\crefformat{equation}{(#2#1#3)}
\crefformat{notation}{(#2#1#3)}
\crefformat{assumption}{(#2#1#3)}
\crefformat{definition}{(#2#1#3)}
\crefformat{proposition}{#2#1#3}
\crefformat{section}{#2#1#3}
\crefformat{item}{#2#1#3}
\crefformat{fig}{#2#1#3}

\DeclareMathOperator{\sign}{sign}

\newcommand{\Rd}{{\R^d}}

\DeclareMathOperator{\Mass}{\nu_{\alpha}}
\DeclareMathOperator{\Masss}{\nu_{\alpha}^{(\boldsymbol{a})}}
\DeclareMathOperator{\mass}{\widetilde{\nu}_{\alpha}}

\newcommand{\R}{\mathbb{R}}
\newcommand{\RR}{\mathcal{R}}
\newcommand{\Ro}{\R_*}

\newcommand{\E}{\mathcal{E}}

\newif\ifdetails
\detailstrue     

\title{Critical fractional Hardy inequalities}

\author[K.~Bogdan]{Krzysztof Bogdan}
\address{Faculty of Pure and Applied Mathematics, Wroc\l{}aw University of Science and Technology, Wyb. Wyspia\'nskiego 27, 50-370 Wroc\l{}aw, Poland.}
\email{krzysztof.bogdan@pwr.edu.pl}

\author[B.~Dyda]{Bart\l{}omiej Dyda}
\address{Faculty of Pure and Applied Mathematics, Wroc\l{}aw University of Science and Technology, Wyb. Wyspia\'nskiego 27, 50-370 Wroc\l{}aw, Poland.}
\email{bartlomiej.dyda@pwr.edu.pl\quad dyda@math.uni-bielefeld.de}

\author[A.~Szczukiewicz]{Antoni Szczukiewicz}
\address{Faculty of Pure and Applied Mathematics, Wroc\l{}aw University of Science and Technology, Wyb. Wyspia\'nskiego 27, 50-370 Wroc\l{}aw, Poland.}
\email{264032@student.pwr.edu.pl}

\date{October 7, 2025}


\begin{document}

\maketitle

\begin{abstract}
   We give a roadmap for the study of critical fractional Hardy inequalities and resolve the one-dimensional cases for Dirichlet and Sobolev–Bregman forms.
\end{abstract}
\section{Introduction}

In what follows, $\alpha \in (0, 2)$. Let us start  with the fractional Laplacian on the real line,
\begin{equation}
    \label{e.fL}
\Delta^{\alpha/2}u(x) := \underset{\epsilon \to 0+}{\lim} \int_{|y-x|>\epsilon} \left(u(y)-u(x)\right) \mathcal {A}_{\alpha}\Mass(x,y) \,dy,\quad x\in \mathbb R.
\end{equation}
Here and later on,
    \begin{equation}
        \label{e.na}
        \Mass(x, y) := |x - y|^{-\alpha - 1},\quad x,y\in \R,    \end{equation}
    $\mathcal {A}_{\alpha}:=
2^{\alpha}\Gamma\big((1+\alpha)/2\big)\pi^{-1/2}/|\Gamma(-\alpha/2)|$, and, of course, $\Gamma$ is the Euler gamma function. The operator is nonlocal, meaning that $\Delta^{\alpha/2}u(x)$ depends on values of $u$ beyond a neighborhood of $x$. The principal value integrals in \eqref{e.fL} converge, e.g., for functions $u\in C_c^{2}(\mathbb R)$, but, 
in fact, we are more interested in the corresponding quadratic form
\begin{equation}\label{definition: quadratic form}
\frac{1}{2} \int_{\R} \int_{\R} (u(x)-u(y))^2 \mathcal {A}_{\alpha}\Mass(x,y) \,dy \,dx,
\end{equation}
which is well defined for every (Borel measurable) $u: \R \to \R$.
Then for all $0<\alpha<1$, $\alpha-1 \leq \beta \leq 0$,
and
$u \in L^2(\R, dx)$, we have the \textit{ground state representation} of \eqref{definition: quadratic form} as
\begin{align}\label{e:p2}
&
\int_{\R} u(x)^2\,\kappa_{-\beta}|x|^{-\alpha}\,dx
 + \frac{1}{2} \int_{\R}\!\int_{\R}
\left(\frac{u(x)}{h(x)}-\frac{u(y)}{h(y)}\right)^2
h(x)h(y) {\mathcal{A}}_{\alpha}\Mass(x,y)
\,dy\,dx,
\end{align}
where $h(x):=|x|^{\beta}$ is the \textit{ground state} and
\begin{equation}\label{e.dkb}
\kappa_{\delta}:=\frac{2^\alpha\Gamma\big(\frac{\delta+\alpha}{2}\big) \Gamma\big(\frac{1-\delta}{2}\big)}{\Gamma\big(\frac{\delta}{2}\big)\Gamma\big(\frac{1-\delta-\alpha}{2}\big)},\quad \delta\in(-\alpha,1),
\end{equation}
see Bogdan, Dyda, and Kim \cite[Proposition 5]{MR3460023}\footnote{Note that \cite{MR3460023} uses a different parametrization $h(x)=|x|^{-\beta}$.} or Frank, Lieb, and Seiringer \cite[Proposition 4.1]{MR2425175}.
To be clear, \cite{MR3460023} and \cite{MR2425175} give ground state representations in  $\R^d$ for $d\in \mathbb N$ and a range of $\alpha$, but
the case $d=1$ considered in this paper yields the  restriction $\alpha<1$ in  \eqref{e:p2}; see however Corollary~\ref{c.La2} and Subsection~\ref{ss.uorm}. 

The \textit{first goal} of this work is to extend the ground state representation 
to the forms 
\begin{equation}\label{definition: quadratic formg}
{\E_k}
{[u]}:= \frac{1}{2} \int_{\R} \int_{\R} (u(x)-u(y))^2 k(x,y) \,dy \,dx,
\end{equation}
with symmetric (Borel measurable) kernels $0\le k\le c\Mass$ homogeneous of order $-\alpha -1$. The resulting  ground state representation is given explicitly in Corollary~\ref{corollary: case p = 2 ground state} below for functions $u\in L^2(|x|^{-\alpha}dx):=L^2(\R,|x|^{-\alpha}dx)$, where, compared to \eqref{e:p2}, we relax the restriction on $\alpha$ and allow for \textit{asymmetric} power functions $h$.

Furthermore, it is known that $\kappa_{\delta}$ 
 attains its  maximum  at
$\delta=(1-\alpha)/2$ and
\begin{equation}\label{e.mkb}
\kappa_{(1-\alpha)/2}=2^\alpha\Gamma\left(\frac{1+\alpha}{4}\right)^2/
\Gamma\left(\frac{1-\alpha}{4}\right)^{2},
\end{equation}
see \cite[Lemma 3.2]{MR2425175} or \cite[p.~237]{MR3460023}.
In particular,
the following Hardy inequality holds
\begin{equation}\label{e:hardy-quad}
\frac{1}{2} \int_{\R} \int_{\R} (u(x)-u(y))^2 \mathcal {A}_{\alpha}\Mass(x,y) \,dy \,dx \geq
\int_{\R} u(x)^2 \,\kappa_{(1-\alpha)/2}|x|^{-\alpha}\,dx
\end{equation}
for all $u \in L^2(\R)$.
The inequality is also known as Hardy-Rellich inequality and  was proved by Herbst \cite[(2.6)]{MR436854}, Beckner \cite[Theorem 2]{MR1254832}, and Yafaev \cite[Theorem 2.9]{MR1717839} in arbitrary dimensions. Herbst and Beckner point out that the \textit{Hardy constant} $\kappa_{(1-\alpha)/2}$ in \cref{e:hardy-quad} is \textit{sharp}, i.e. the largest possible for the given class of functions $u$. See also \cite[Subsection~2.1]{MR2425175}.

In fact,
the \textit{Hardy weight} $\kappa_{(1-\alpha)/2}|x|^{-\alpha}$ in \cref{e:hardy-quad} cannot be increased on any set of positive Lebesgue measure. 
In such cases, we say that the weight is \textit{critical}, as is the form defined by the difference of the two sides of \eqref{e:hardy-quad}.
The criticality of $\kappa_{(1-\alpha)/2}|x|^{-\alpha}$ is proved in Takeda and Uemura \cite[Example 5.9]{MR4586808};  the criticality of quadratic Schr\"odinger forms closely related to \eqref{definition: quadratic form} is resolved in Miura \cite{MR4585113}.
We note that the   sharpness of Hardy constants has been widely studied in the literature; see, e.g.,  Kufner, Maligranda, and Persson \cite{MR2256532}, Abdellaoui and Bentifour \cite{MR3626031}, Dyda and Kijaczko \cite{MR4705882}, Kijaczko and Lenczewska \cite{MR4720167}, Boggarapu, Roncal, and Thangavelu \cite{MR3962190}, or Gupta \cite{MR4531775}. In contrast, 
the criticality of Hardy weights is a matter of recent progress. For the Laplacian in dimensions $d\ge 3$, see Devyver, Fraas, and, Pinchover \cite[Example 3.1]{MR3170212}.
For the \textit{discrete case}, see Keller, Lenz, and Wojciechowski \cite[Chapter 9]{MR4412542}, Keller and Nietschmann \cite{MR4612316}, Fischer \cite{MR4768506}, and Fischer \cite{MR4597627}. We also note that in the latter works, the criticality of weights or forms is obtained via the ground state representation.

 Accordingly, our \textit{second goal} is to give Hardy inequalities and prove criticality results for the forms given by \eqref{definition: quadratic formg}, that is, with rather general symmetric kernels $k$ homogeneous of degree $-\alpha-1$ and bounded by a multiple of $\Mass$. These are given in Corollary~\ref{corollary: case p = 2 critical weight}. 
 Notably, the asymmetric analogues of the power functions \( h \) are parametrized not only by \( \beta \), but also by an asymmetry parameter \( s \in (0, \infty) \).
They allow for \textit{balancing} the (restrictions of the) Hardy weights on the right and left half-lines of $\R$: one can be increased at the expense of the other, which may even become negative.
The proof of criticality is based on constructing a suitable sequence of functions $u$ approximating the ground state, which are then used in the ground state representation; see \eqref{e.da}.

 Let us remark that this work started with questions on  the Servadei--Valdinoci form
\begin{align}\label{e.SVf}
    S[u] := \frac{1}{2} \iint\limits_{\substack{x > 0 \\ \text{or } y > 0}} \bigl(u(x) - u(y)\bigr)^2\, \mathcal{A}_\alpha \nu_\alpha(x,y)\, dx\, dy.
\end{align}
Of course, $S[u]$ is a special case of \eqref{definition: quadratic formg} with $
k(x,y) = \mathcal{A}_\alpha\, |x - y|^{-\alpha - 1} \left(1 - \mathbf{1}_{\{x < 0\}}\, \mathbf{1}_{\{y < 0\}} \right)
$.
Such quadratic forms, defined via integration on $\Rd\times \Rd \setminus D\times D$ for rather general domains $D$, recently appeared in the theory of nonlocal PDEs. They improve the variational approach to solving the Dirichlet or Neumann problems on domains for the fractional Laplacian and other nonlocal operators. This is so because the presumed finiteness of the form $S$ puts minimal assumptions on the regularity of solutions outside of $D$; see Bogdan, Grzywny, Pietruska-Pałuba, and Rutkowski \cite{MR4088505} or just the discussion in \cite[Introduction]{MR4589708} by the same authors.
The forms were introduced  in Servadei and Valdinoci \cite{MR2879266, MR3002745}.
Then they were used in Ros-Oton \cite[(3.1)]{MR3447732} and Dipierro, Ros-Oton and Valdinoci \cite[p. 379]{MR3651008}; see also Felsinger, Kassmann and Voigt \cite[Definition 2.1 (ii)]{MR3318251}.
Recently, the specific form \eqref{e.SVf} appeared in Bogdan, Fafuła, and Sztonyk \cite{BogdanFafulaSztonyk2025} as the Dirichlet form of a certain Markov process on \( \mathbb{R} \).
 Moreover, in \cite[Theorem 1.2.1]{BogdanFafulaSztonyk2025}, the authors gave a Hardy inequality for $S$ with an asymmetric weight on $\R$. In fact, the \textit{initial goal} of this work was to verify whether the Hardy constant  in \cite{BogdanFafulaSztonyk2025} is sharp and whether the Hardy weight can be made symmetric.
 Both questions are answered in the affirmative in Subsection~\ref{ss.SV}, as applications of the results presented in Section~\ref{section : main results}.

Last, but not least, we are interested in counterparts of these developments in the setting of $L^p$. Namely,
for  $p \in (1, \infty)$, $k$ as above,  and (Borel measurable) $u: \R \to \R$,
we define the Sobolev--Bregman form, or the $p$-form,
\begin{equation}\label{e:dEp}
\E_{k, p}[u] :=
\frac{1}{2}
\int_{\R} \int_{\R} (u(x)-u(y)) (u(x)^{\langle p - 1 \rangle}  -u(y)^{\langle p - 1 \rangle} )\,k(x,y) \,dy \,dx.
\end{equation}
Here and below, we use the notation
\[
a^{\langle k \rangle} := |a|^k\, \sign(a), \quad a, k \in \mathbb{R},
\]
with the convention that \( 0^{\langle k \rangle} := 0 \). Of course, $\E_{k, 2}=\E_{k}$.
Sobolev--Bregman forms more general than \eqref{e:dEp} are exact \( L^p \)-analogues of Dirichlet forms, which play a central role in the \( L^2 \)-theory of Markov semigroups; see Fukushima, Oshima, and Takeda \cite{MR1303354}, Ma and Röckner \cite{MR1214375} or Chen and Fukushima \cite{MR2849840}.
In particular, Gutowski and Kwaśnicki \cite{MR4885983} provides a Beurling-Deny decomposition for Sobolev--Bregman forms of symmetric Markov semigroups and  Bogdan, Jakubowski, Lenczewska, and Pietruska-Pałuba \cite[Theorem 3]{MR4372148} gives an application to $L^p$-norms of certain Feynman-Kac semigroups, which also motivate our development. See also \cite[Subsection 1.3]{MR4372148} and \cite[Introduction]{MR4589708} for  historical context and further motivation.

Accordingly, the \textit{main goal} of the paper is to combine all the above threads in the generality of the Sobolev--Bregman forms \eqref{e:dEp}, including ground state representations and criticality of the resulting Hardy weights in $L^p(|x|^{-\alpha}dx):=L^p(\R,|x|^{-\alpha}dx)$, for all $p\in (1,\infty)$. This is accomplished in Theorems \ref{theorem: Hardy Identity} and \ref{theorem: Hardy Inequality}. 
We also allow for 
balancing of the critical Hardy weights
and
an extra additive term on the right-hand side of critical Hardy inequalities\footnote{We call the result \emph{improved Hardy inequalities}.} and
we
resolve in detail special important cases.
The results include the quadratic case \( p = 2 \) discussed above,  both in the statements and proofs, but they are not derived as consequences of that case.
%
Note that for $1<p<\infty$, the ground state representation and Hardy inequalities with sharp Hardy \textit{constants} 
for Sobolev--Bregman forms were already established for the fractional Laplacian on $\mathbb{R}^d$ 
\cite[Theorems~1 and~2]{MR4372148} and for the fractional Laplacian on the half-space 
\cite[Lemma~2.1 and Theorem~1.1]{MR4720167}. 
Apart from these overlaps, our results are new — in fact, even for $p=2$ — in the generality of the kernels $k$ considered here. The results are also definitive; of further interest, of course, are possible extensions to other classes of kernels $k$, especially in higher dimensions, and applications similar to \cite{MR4372148}, but they should follow the roadmap outlined below.

Before closing the Introduction, we should point out that the Sobolev  integral forms,
\begin{equation}\label{e.ddpf}
\iint_{\mathbb{R}^d \times \mathbb{R}^d} \frac{|u(x) - u(y)|^p}{|x - y|^{d + \alpha p/2}} \, dx \, dy,
\end{equation}
ubiquitous in Functional Analysis and  PDEs (see Di Nezza, Palatucci, and Valdinoci \cite{MR2944369}),
are fundamentally different from the Sobolev--Bregman forms; 
see \eqref{e.cHk} and \cite{MR4885983}.
In Frank and Seiringer \cite{MR2469027}, Fischer  \cite{MR4597627}, Dyda and Kijaczko \cite{MR4708667, MR4705882}, 
ground state representations are given for Sobolev forms, but the representations are, in fact, \textit{inequalities} for $p\neq 2$; see \cite[Proposition 2.3]{MR2469027}, \cite{MR4597627}, \cite[Theorem 1.2, 1.3]{MR4705882}, and  \cite[Theorem 2]{MR4708667}. In contrast, our ground state representations  for 
Sobolev--Bregman forms are exact identities for all $p\in (1,\infty)$. 


The structure of the paper is as follows. In Section~\cref{sec:Pre}, we provide necessary definitions. In Section~\cref{section : main results}, we present precise statements of the results mentioned above, with additional discussion of the case $p=2$ in Subsection~\cref{ss.p2} and proofs given in Section~\cref{sec: proofs}.
 In Section~\cref{sec:Example}, we give auxiliary results and applications, in particular to the Servadei–-Valdinoci form. 

\section*{Acknowledgments}
We thank Elvise Berchio, Matthias Keller, Yehuda Pinchover, and Luz Roncal for discussions and the invitation to the Oberwolfach Mini-Workshop Hardy Inequalities in Discrete and Continuum Settings \cite{OberwolfachReport2025-14}, which greatly stimulated our work on this paper and where some of the results were presented.
We also thank Daman Fafuła and Paweł Sztonyk for discussions on the domain of the Hardy inequality for the Servadei--Valdinoci form, Shubham Gupta for discussions on the Sobolev--Bregman forms, and Florian Fisher for inspiring talks and encouragement to use the ground state representation to prove criticality. Krzysztof Bogdan was  supported by the Opus grant 2023/51/B/ST1/02209 of National Science Center, Poland. Antoni Szczukiewicz was supported as a student by Wrocław University of Science and Technology and the Mathematisches Forschungsinstitut Oberwolfach.

\section{Preliminaries}
\label{sec:Pre}

We use "$:=$" to indicate definitions, e.g.,
$a \wedge b := \min \{ a, b\}$ and $a \vee b := \max \{ a, b\}$.
For functions $f$ and $g$, we write $f\lesssim g$ to assert that there is a number $c\in (0,\infty)$, which we then call a \textit{constant}, such that $f(x)\leq c g(x)$ for all the considered arguments $x$. All the sets, functions, and measures in the paper are assumed Borel. 
For an open subset $D\subset \mathbb R$, we let $C_c(D)$ be the space of all the continuous functions with compact support in $D$, but we abbreviate $C_c(0, \infty):=C_c((0, \infty))$. We let $\Ro := \R \setminus\{0\}$. Further notation is introduced as we proceed. Recall that $\alpha \in (0, 2)$ and let us reiterate the assumptions on the kernels $k$ in \eqref{definition: quadratic formg} and \eqref{e:dEp}.


\begin{definition}
We denote by \( K_\alpha \) the set of all the functions \( k : \mathbb{R}^2 \to [0, \infty] \) such that \( k(x, y) = k(y, x) \) for all \( x, y \in \mathbb{R} \) (symmetry), \( k(rx, ry) = r^{-\alpha - 1} k(x, y) \) for all \( x, y \in \mathbb{R} \), \( r > 0 \) (homogeneity of order \( -\alpha - 1 \)), and \( k \lesssim \Mass \) (bound by a constant multiple of $\Mass$).
\end{definition}

Below, we say that numbers \( p, q\)  are \textit{conjugate exponents} if \(p,q\in (1, \infty) \) and \( \frac{1}{p} + \frac{1}{q} = 1 \). 
Of course, the equality is equivalent to \( q = \frac{p}{p-1} \) and to \( (p-1)(q-1) = 1 \). 

For $p\in (1,\infty)$, we also consider the \textit{Bregman divergence}:
\begin{equation}
F_p(a,b) := |b|^p - |a|^p - pa^{\langle p - 1 \rangle} (b-a), \quad a, b \in \R.
\end{equation}
For instance, $F_2(a,b)= (b-a)^2$.
Since $F_p$ is the second-order Taylor remainder of the convex function $\R\ni x\mapsto |x|^p$, we get $F_p(a,b) \geq 0$.
The
\textit{symmetrization} of $F_p$ is
\begin{equation}\label{e:fp}
\frac12 (F_p(a,b) + F_p(b,a)) =\frac{p}{2} (b-a)(b^{\langle p - 1 \rangle} -a^{\langle p - 1 \rangle} ).
\end{equation}
Therefore, by the symmetry of $k$, for all $u$, we can rewrite the integral in \eqref{e:dEp} as
\begin{equation}\label{e.dEp2}
\E_{k, p}[u] =
 \frac{1}{p}\int_{\R}\int_{\R}F_p(u(x), u(y))k(x, y)\,dx\,dy.
\end{equation}
The integrands in \eqref{e.dEp2} are nonnegative so the integrals are well defined.
In fact, we can analyze $\E_{k, p}$ using  $\E_k=\E_{k, 2}$ and
\begin{equation}\label{e.cHk}
4(p-1)p^{-2}(b^{\langle p/2\rangle}-a^{\langle p/2\rangle})^2\leq (b-a)(b^{\langle p-1\rangle}-a^{\langle p-1\rangle})\leq 2(b^{\langle p/2\rangle}-a^{\langle p/2\rangle})^2.
\end{equation}
The inequality \eqref{e.cHk} 
holds true for all $p\in (1,\infty)$ and $a,b\in \R$, see, e.g., Liskevich, Perelmuter, and Semenov \cite[Lemma 2.1]{MR1407327} or \cite{MR4885983}. It does \textit{not}, however, yield  optimal Hardy weights for the Sobolev--Bregman forms from the case $p=2$; see \cite{MR4372148}.  Instead, in this paper we rely on the algebraic identity \eqref{equation: rownosc liczbowa} below. Our approach gives a road map for ground-state representations and Hardy inequalities for general Sobolev-Bregman forms.
We further refer to Bogdan, Pietruska-Pa{\l}uba, and Gutowski \cite[(2.12)-(2.14)]{MR4851904} for a discussion of various estimates of $F_p$, to Bogdan, Grzywny, Pietruska-Pa{\l}uba, and Rutkowski \cite{MR4589708} for references to applications of Bregman divergence in analysis, statistical learning, and optimization, and to Bogdan, Pietruska-Pa{\l}uba, and Kutek \cite{bogdan2024bregmanvariationsemimartingales} for a martingale connection. 

In the paper, the (critical) Hardy weights are understood  as follows.
\begin{definition}
Let $p \in (1, \infty)$, $k \in K_{\alpha}$ and $w : \R \xrightarrow[]{} \R$ be locally integrable on $\Ro$. We say that $w$ is $(k, p)$-Hardy weight if 
\begin{align}
\label{equation: general inequality}
    \E_{k, p}[u] \geq \int_{\R}|u(x)|^pw(x)dx,\qquad u \in C_c(\Ro).
\end{align}
If, moreover, $w$  cannot be increased in \eqref{equation: general inequality} on any set of positive Lebesgue measure 
then we say that $w$ is a critical $(k, p)$-Hardy weight. Similarly, we say that $w : (0, \infty) \rightarrow \R$ is a $(k, p)$-Hardy weight on the half-line if
\begin{align}
\label{equation: general inequality half-line}
    \frac{1}{p}\int_{0}^{\infty}\int_{0}^{\infty}F_p(u(x), u(y))k(x, y)\,dx\,dy \geq \int_{0}^{\infty}|u(x)|^pw(x)dx,\qquad u \in C_c(0, \infty),
\end{align}
and we say $w$
is a critical $(k, p)$-Hardy weight on the half-line if additionally it cannot be 
increased in \eqref{equation: general inequality half-line} on any set of positive Lebesgue measure.
\end{definition}
 Note that we do not require $w$ above to be nonnegative. In fact, some critical $(k, p)$-Hardy weights are not nonnegative; see Theorem \cref{theorem: sign}.

Our Hardy weights will be expressed in terms of the following quantities.\footnote{The computations motivating the definitions will appear in the proof of \eqref{equation: rownosc Kw} (Lemma~\ref{lemma: rownosc Kw}).} 
\begin{definition}
\label{definition: gamma C}
For $\rho\in (-1, \alpha)$, $r \in (0, \infty)$, and $k \in K_\alpha$, we define
\begin{align}
\begin{split}\label{e.gamma}
\gamma_{+}(k, \rho) &:= \int_{0}^{1}(1 - t^\rho)(1 - t^{\alpha - 1- \rho})k(1, t)dt,
\\
    \gamma_{-}(k, \rho) &:= \int_{0}^{1}(1 - t^\rho)(1 - t^{\alpha - 1- \rho})k(-1, -t)dt,
\end{split}
\\
\begin{split}
\label{equation: stałe rownosc hardyego}
    C_{+}(k, r, \rho) &:= \gamma_{+}(k, \rho) + \int_{0}^{\infty}k(1, -t)dt - r\int_{0}^{\infty}t^\rho k(1, -t)dt, 
    \\
    C_{-}(k, r, \rho) &:= \gamma_{-}(k, \rho) + \int_{0}^{\infty}k(-1, t)dt - \frac{1}{r}\int_{0}^{\infty}t^\rho k(-1, t)dt.
\end{split}
\end{align}
\end{definition}
For instance, $-\gamma_+(\nu_\alpha,\rho)$ is the constant $\gamma(\alpha,\rho)$ in \cite[(5.2)]{MR2006232}, see  \eqref{equation: gamma def}. 
\begin{definition}
\label{definition: D}
For $\beta \in (-p, p\alpha) \cap \left(-q, q\alpha \right)$, $s \in (0, \infty)$, conjugate exponents $p, q$, and $k \in K_\alpha$, we let
\begin{align}
\begin{split}\label{e.dD}
    D_+(k, p, s, \beta) := \frac{1}{p}C_{+}\left(k, s^{\frac{1}{q}}, \frac{\beta}{q}\right) + \frac{1}{q}C_{+}\left(k, s^{\frac{1}{p}}, \frac{\beta}{p}\right),
    \\
    D_-(k, p, s, \beta) := \frac{1}{p}C_{-}\left(k, s^{\frac{1}{q}}, \frac{\beta}{q}\right) + \frac{1}{q}C_{-}\left(k, s^{\frac{1}{p}}, \frac{\beta}{p}\right).
\end{split}
\end{align}
\end{definition}

The following will be our Hardy weights
\begin{align}\label{e.dV}
    w_{k, p, s, \beta}(x) := |x|^{-\alpha} \begin{cases}
    D_+(k, p, s, \beta), \quad x > 0,
    \\
    D_-(k, p, s, \beta), \quad x < 0.
    \end{cases}
\end{align}
Thus, $D_+$ and $D_-$ determine how \textit{balanced} is the Hardy weight.\footnote{The computations motivating  \eqref{e.dD} and \eqref{e.dV} will appear in the proof of Theorem~\cref{theorem: Hardy Identity}.}

As we shall see later on, the case of $\beta=\alpha-1$ is of special importance.
 Figures \ref{fig: D alpha=1.4} and \ref{fig: D alpha = 1} show $D_+(k, p, s, \alpha-1)$ and $D_-(k, p, s, \alpha-1)$ as functions of $s$ for the kernel $\mass$ of the Servadei--Valdinoci form \eqref{e.SVf} and $p = 3$; see the discussion  in Section \cref{ss.SV}.
\begin{figure}[H]
\begin{tikzpicture}[scale=1]
\begin{axis}[
    axis lines = middle,
    xtick={0, 1, 2},
    ytick={-0.4, 0.4},
    xticklabels={0, 1, 2},
    yticklabels={-0.4, 0.4},
    ymin=-0.4,
    ymax=0.4,
    xmin=0,
    xmax=2,
    xlabel={$s$},
    ylabel={},
    grid=none,
    legend pos=south east,
    title={},
    width=14cm,
    height=10cm,
    clip mode=individual,
    restrict y to domain=-10:10,
    unbounded coords=jump,
    samples=100,
]

\addplot[
    color=black,
    thick,
    mark=none,
    domain=0.01:2  
] table [col sep=comma, x=x, y=y] {Dp0.csv};
\addlegendentry{$D_+$}

\addplot[
    color=black,
    thick,
    mark=none,
    dashed,
    domain=0.01:2  
] table [col sep=comma, x=x, y=y] {Dm0.csv};
\addlegendentry{$D_-$}

\end{axis}
\end{tikzpicture}
\caption{Plots of  $s \mapsto D_+(\mass, 3, s, \alpha-1)$ and $s \mapsto D_-(\mass, 3, s, \alpha-1)$ for $\alpha=1$.}
\label{fig: D alpha = 1}
\end{figure}

The following functions are pivotal. They will be used for Doob conditioning in the proof of Theorem~\ref{theorem: Hardy Identity}, and, as we shall see, they also yield the Hardy weights and  constants $C_+$, $C_-$, $D_+$, and $D_-$.

\begin{definition} For $p \in (1, \infty)$, $\beta \in \R$, and $s > 0$, we define
\begin{align}
    h_{\beta, s}(x) &:= |x|^{\beta}\begin{cases}
        1, &\quad x > 0,
        \\
        s, &\quad x<0,
    \end{cases}\label{e.dhbs}
    \\
    h_{p, \beta, s} &:= h_{\beta, s}^{\frac{1}{p}} = h_{\frac{\beta}{p}, s^{\frac{1}{p}}}.
\end{align}
\end{definition}
Of course, $h_{\beta,s}$ is symmetric if and only if $s=1$. Note that if $p$ and $q$ are conjugate exponents then $h_{p, \beta, s}^{p-1} = h_{q, \beta, s}$.

\section{Main results}
\label{section : main results}
Our main results
are presented below in this section. Their proofs are either evident or postponed to Section \cref{sec: proofs}.
\begin{theorem}
\label{theorem: Hardy Identity}
For $\alpha \in (0, 2)$, conjugate exponents $p$ and $q$, $\beta \in (-p, p\alpha) \cap \left(-q, q\alpha \right)$, $s \in (0, \infty)$, $k \in K_\alpha$, and $u \in L^p(|x|^{-\alpha}dx)$, the following ground state representation holds:
\begin{align}
\label{equation: ground state representation}
\begin{split}
        \E_{k, p}[u] &= \int_{\R}|u(x)|^pw_{k, p, s, \beta}(x)dx
    \\
    &+  \frac{1}{p}\int_{\R}\int_{\R}F_p \left( \frac{u(x)}{h_{p, \beta, s}(x)}, \frac{u(y)}{h_{p, \beta, s}(y)}\right) h_{q, \beta, s}(x)h_{p, \beta, s}(y) k(x, y)\,dx\,dy.
\end{split}
\end{align}
\end{theorem}
Since $F_p\ge 0$ and $h_{q,\beta,s}\ge 0$, we immediately obtain the following Hardy inequality. 
\begin{corollary}\label{c.Hip}
 For $\alpha \in (0, 2)$, $p \in (1, \infty)$, $\beta \in (-p, p\alpha) \cap \left(-q, q\alpha \right)$, $s \in (0, \infty)$, $k \in K_\alpha$, 
\begin{align}
\label{equation: hardy inequality V}
\begin{split}
        \E_{k, p}[u] \geq \int_{\R}|u(x)|^pw_{k, p, s, \beta}(x)dx, \quad u \in L^p(|x|^{-\alpha}dx).
\end{split}
\end{align}
\end{corollary}
As noted above, the $\beta=\alpha-1$ plays a special role in this theory, so we denote
\begin{align}
    W_{k, p, s} := w_{k, p, s, \alpha -1}.
\end{align}
The following result complements Corollary~\ref{c.Hip}.
\begin{theorem}
\label{theorem: Hardy Inequality}
    If $\alpha \in (0, 2)$, $p \in (1, \infty)$, $s \in (0, \infty)$, and $k \in K_\alpha$, then
$W_{k, p, s}$ is a critical $(k, p)$-Hardy weight. 
\end{theorem}
Next we  essentially specialize to the case where $k(x,y)=0$ for $x<0$ or $y<0$.
\begin{corollary}\label{c.pnH} For all $\alpha \in (0, 2)$, $p \in (1, \infty)$, $k \in K_\alpha$, and $u \in L^p((0,\infty),x^{-\alpha}dx)$,
\begin{align}
\label{equation: pos -line}
    &\frac{1}{p}\int_{0}^{\infty}\int_{0}^{\infty}F_p(u(x), u(y))k(x, y)\,dx\,dy \geq \gamma_{+}\left(k, \frac{\alpha - 1}{p}\right)\int_{0}^{\infty}\frac{|u(x)|^p}{x^{\alpha}}dx.
\end{align}
Moreover, $\gamma_{+}\left(k, \frac{\alpha - 1}{p}\right)x^{-\alpha}$ is a critical $(k, p)$-Hardy weight on the half-line.
\end{corollary}
Corollary \cref{c.pnH} follows directly from applying Theorem \cref{theorem: Hardy Inequality} to the kernel $$k(x, y)\mathds{1}_{(0, \infty)}(x)\mathds{1}_{(0, \infty)}(y).$$

In the same vein, we say that \( k \in K_\alpha \), with \( \alpha \in (0, 2) \), is \emph{reducible} if \( k(-1, t) = 0 \) for \text{a.e.} \( t \in (0, \infty) \). It follows easily from the homogeneity of \( k \in K_\alpha \) that \( k \) is reducible if and only if \( k(x, y) = 0 \) for \text{a.e.} \( (x, y) \in \mathbb{R}^2 \) such that \( xy < 0 \). 

We say that \( k \in K_\alpha \) is \emph{irreducible} if it is not reducible.

\begin{theorem}
\label{theorem: sign}
For $p \in (1, \infty)$ the following assertions hold.
\begin{enumerate}
    \item[(a)]
    \label{item: sign 1}
    If $k \in K_1$ then $0$ is a critical $(k, p)$-Hardy weight.
    \item[(b)]
    \label{item: sign 2}
If \( \alpha \in (0, 2) \setminus \{1\} \) and \( k \in K_\alpha \) is reducible but not zero a.e., then \( W_{k,p,s} \) is independent of \( s \), nonnegative, and strictly positive on at least one of the half-lines.

    \item[(c)]
    \label{item: sign 2b}
    If $\alpha \in (0, 2) \setminus \{1\}$  and $k \in K_\alpha$ is irreducible, then there exists a nonempty open interval $S \subset (0, \infty)$, such that $W_{k, p, s} > 0$ for all $s \in S$. Furthermore, there exist $s_0 > 0$ and  $D > 0$ such that $W_{k, p, s_0}(x) = D|x|^{-\alpha}$ for all $x\in \R$.
\end{enumerate}
\end{theorem}
Of course, (c) above means 
that $D_+(k, p, s_0, \alpha - 1) = D_-(k, p, s_0, \alpha - 1)$ in \eqref{e.dV}; this is illustrated by Figure~\ref{fig: D alpha=1.4}. 
 \begin{figure}[H]
\begin{tikzpicture}[scale=1]
\begin{axis}[
    axis lines = middle,
    xtick={0, 1.20026, 1, 2},
    ytick={-0.3, 0.027093, 0.3},
    xticklabels={0, $s_0$, 1, 2},
    yticklabels={-0.3, $D$, 0.3},
    ymin=-0.3,
    ymax=0.3,
    xmin=0,
    xmax=2,
    xlabel={$s$},
    ylabel={},
    grid=none,
    legend pos=south east,
    title={},
    width=14cm,
    height=10cm,
    clip mode=individual,
    restrict y to domain=-10:10,
    unbounded coords=jump
]

\addplot[
    color=black,
    thick,
    mark=none,
] table [col sep=comma, x=x, y=y] {Dp.csv};
\addlegendentry{$D_+$}

\addplot[
    color=black,
    thick,
    mark=none,
    dashed,
    domain=0.01:2
] table [col sep=comma, x=x, y=y] {Dm.csv};
\addlegendentry{$D_-$}

\addplot[
    black,
    semithick,
    domain=0:0.027093,
    samples=2,
    dotted,
] ({1.20026}, x);

\addplot[
    black,
    semithick,
    domain=0:1.20026,
    samples=2,
    dotted,
] (x, {0.027093});

\end{axis}
\end{tikzpicture}
\caption{Plots of  $s \mapsto D_+(\mass, 3, s, \alpha-1)$ and $s \mapsto D_-(\mass, 3, s, \alpha-1)$ for $\alpha=1.4$.}
\label{fig: D alpha=1.4}
\end{figure}
\subsection{Improved critical Hardy inequality}
The title of this subsection may seem like a misnomer, but it refers to an additional \textit{additive} term in our critical Hardy inequalities, obtained by optimizing the right-hand side of \eqref{equation: hardy inequality V} over $s \in (0,\infty)$. 
\begin{definition} For $\alpha \in (0, 2)$, $p \in (1, \infty)$,  $k \in K_\alpha$, and 
function $u$, we let
\begin{align*}
    l_{k, p} &:= \int_{0}^{\infty}t^{\frac{\alpha-1}{p}}k\left(1,-t\right)dt,
    \\
    I_{\alpha, p, +}(u) &:= \int_{0}^{\infty}\frac{|u(x)|^p}{|x|^{\alpha}}dx,
    \\
    I_{\alpha, p, -}(u) &:= \int_{-\infty}^{0}\frac{|u(x)|^p}{|x|^{\alpha}}dx.
\end{align*}
\end{definition}

The following result yields a nonnegative term $\mathcal{G}_{k,p}(u)$ that can be added to the right-hand side in \cref{equation: hardy inequality V} even with the critical $(k,p)$-Hardy weight $W_{k,p,1}$.

\begin{proposition}
\label{proposition: extra term}
For $\alpha \in (0, 2)$, $p \in (1, \infty)$, $k \in K_\alpha$, and $u \in L^p(|x|^{-\alpha}dx)$,
\begin{align}
\label{equation: corrected inequality}
    &\sup_{s>0}\int_{\R} |u(x)|^p W_{k,p,s}(x)\,dx = \int_{\R}|u(x)|^p W_{k, p, 1}(x)\,dx + \mathcal{G}_{k, p}(u),
\end{align}
where
\begin{align*}
    \mathcal{G}_{k, p}(u) := l_{q}\left(\frac{I_{+}}{p}+\frac{I_{-}}{q}-I_{+}^{\frac{1}{p}}I_{-}^{\frac{1}{q}}\right)+\ l_{p}\left(\frac{I_{+}}{q}+\frac{I_{-}}{p}-I_{+}^{\frac{1}{q}}I_{-}^{\frac{1}{p}}\right)\geq 0,
\end{align*}
and
\begin{align*}
    l_p := l_{k, p}, \quad l_q := l_{k, q}, \quad I_+ := I_{\alpha, p, +}(u), \quad I_- := I_{\alpha, p, -}(u).
\end{align*}
\end{proposition}
Note that $\mathcal{G}_{k,p}(u)=0$ if $u$ is symmetric so this terms quantifies asymmetry of $u$.
\begin{corollary}\label{c.pnH}
For $\alpha \in (0, 2)$, $p \in (1, \infty)$, $k \in K_\alpha$, and $u \in L^p(|x|^{-\alpha}dx)$,
\begin{align}
\label{e.iHi} 
    \E_{k, p}[u] &\geq 
 \int_{\R}|u(x)|^p W_{k, p, 1}(x)\,dx + \mathcal{G}_{k, p}(u).
\end{align}    
\end{corollary}
The above \textit{improved} critical Hardy inequality directly follows from Proposition \cref{proposition: extra term}. We thus get a nontrivial lower bound even in the case of $\alpha=1$:
\begin{align*}
    \E_{k, p}[u] \geq \mathcal{G}_{k, p}(u),
\end{align*}
where $\mathcal{G}_{k, p}(u) > 0$ if
  $I_{\alpha, p, +}(u) \neq   I_{\alpha, p, -}(u)$ and $k$ is irreducible; otherwise, $\mathcal{G}_{k, p}(u) = 0$.

\subsection{The case of $p = 2$}\label{ss.p2} In this subsection, we present our main results in the special case $p=2$, where the notation can be simplified.
 For $\alpha \in (0, 2)$, $k \in K_\alpha$, $\rho \in (-1, \alpha)$, and $r \in (0, \infty)$, we let
\begin{align*}
    m_{k, r, \rho}(x) := w_{k,2,r^2,2\rho}(x) = |x|^{-\alpha}\begin{cases}
        C_+(k, r, \rho), &\quad x \geq 0,
        \\
        C_-(k, r, \rho), &\quad x < 0.
    \end{cases}
\end{align*}
We first give the resulting ground state representation and criticality.
\begin{corollary}
\label{corollary: case p = 2 ground state}
For $\alpha \in (0, 2)$, $\rho \in (-1, \alpha)$, $r \in (0, \infty)$, and $k \in K_\alpha$,
\begin{align}
\begin{split}
        \E_{k}[u] &=  \int_{\R}u^2(x)m_{k, r, \rho}(x)dx
    \\
    &+ \frac{1}{2}\int_{\R}\int_{\R} \left( \frac{u(x)}{h_{\rho, r}(x)} - \frac{u(y)}{h_{\rho, r}(y)}\right)^2 h_{\rho, r}(x)h_{\rho, r}(y) k(x, y)\,dx\,dy, \quad u \in L^2(|x|^{-\alpha}dx).
\end{split}
\end{align}
\end{corollary}

\begin{corollary}
\label{corollary: case p = 2 critical weight}
For $\alpha \in (0, 2)$, $r \in (0, \infty)$, and $k \in K_\alpha$,
\begin{align*}
    M_{k, r} := m_{k, r, \frac{\alpha - 1}{2}}
\end{align*}
is a critical $(k, 2)$-Hardy weight.
    
\end{corollary}
We next propose a more explicit formula for the constant $D$ in Theorem \cref{theorem: sign} for $p=2$.
\begin{corollary}
\label{corollary: case p = 2 constant}
For $\alpha \in (0, 2)$ and irreducible $k \in K_\alpha$,
\begin{align*}
    M_{k, r_0}(x) = D|x|^{-\alpha},
\end{align*}
where
\begin{align*}
    r_0 &:= \frac{a}{2d} + \sqrt{1 + \left(\frac{a}{2d}\right)^2},
    \\
    D &:= \frac{1}{2}\left(b - \sqrt{a^2 + 4d^2} \right),
\end{align*}
with
\begin{align*}
    a &:= \int_{0}^{1}\left(1-t^{\frac{\alpha-1}{2}}\right)^2(k(1, t) - k(-1, -t))dt + \int_{0}^{\infty}(1-t^{\alpha-1})k(1, -t)dt,
    \\
    b &:= \int_{0}^{1}\left(1-t^{\frac{\alpha-1}{2}}\right)^2(k(1, t) + k(-1, -t))dt + \int_{0}^{\infty}(1 + t^{\alpha-1})k(1, -t)dt,
    \\
    d &:=\int_{0}^{\infty}t^{\frac{\alpha-1}{2}}k(1, -t)dt.
\end{align*}
\end{corollary}
In passing, we remark that
class \(L^p(|x|^{-\alpha}\,dx)\) used in our paper
appears to be optimal 
and more natural than, say, \(L^p(dx)\). 
For instance, if \(u \in L^p(|x|^{-\alpha}\,dx)\), then the first term on the right-hand side 
of \eqref{equation: hardy inequality V} is finite. 
Moreover, the use of \(L^p(|x|^{-\alpha}\,dx)\) allows us to cover all \(\alpha \in (0,2)\) 
in our results, in contrast to \eqref{e:p2}. 
Accordingly, the next statement should be compared with \eqref{e:p2}.

\begin{corollary}\label{c.La2}
If $\alpha \in (0, 2)$ then
$\kappa_{(1-\alpha)/2}|x|^{-\alpha}$ is a critical Hardy weight in \eqref{e:hardy-quad}, and \eqref{e:p2} holds true for all
$\beta \in (-1, \alpha)$
and
$u \in L^2(|x|^{-\alpha}dx)$.
\end{corollary}
 In this connection we also refer to
 \cite[Proposition 5]{MR3460023},
\cite[Subsection 3.1]{MR1717839}, and 
\cite[Theorem 1.2.1, Proposition 5.3.7, and Proposition 5.3.9]{BogdanFafulaSztonyk2025}. See also the discussion in Subsection~\ref{ss.uorm}

\section{Proofs}
\label{sec: proofs}
\subsection{Ground state representation}
Of course,
\begin{align*}
\sign((1 - t^\rho)(1 - t^{\alpha - 1- \rho})) = \sign(\rho(\alpha - 1- \rho)), \quad t \in (0, 1), \ \alpha, \rho \in \R.
\end{align*}
We first assert the finiteness of $C_{+}(k, r, \rho)$ and $C_{-}(k, r, \rho)$.
\begin{lemma}
For $\alpha \in (0, 2)$, $\rho \in (-1, \alpha)$, $r \in (0, \infty)$, and $k \in K_\alpha$, we have $|C_{+}(k, r, \rho)|$, $|C_{-}(k, r, \rho)| < \infty$.
\end{lemma}

\begin{proof}
It is sufficient to show that each integral appearing in \cref{equation: stałe rownosc hardyego} is finite. Using the inequality $k \lesssim \Mass$ and considering the cases $t \in (0, \frac{1}{2})$ and $t \in [\frac{1}{2}, 1)$, we obtain
\begin{align*}
    |\gamma_{\pm}(k, \rho)| \leq & c_1\int_{0}^{1}|(1 - t^\rho)(1 - t^{\alpha - 1- \rho})|(1 - t)^{-\alpha - 1}dt 
    \\
    \leq & c_2\int_{0}^{\frac{1}{2}}(1 + t^\rho + t^{\alpha - 1- \rho} + t^{\alpha - 1})(1 - t)^{-\alpha - 1}dt + c_3\int_{\frac{1}{2}}^{1}(1 - t)^{1-\alpha} < \infty.
\end{align*}
Furthermore,
\begin{align*}
    &\int_{0}^{\infty}k(1, -t)dt, \int_{0}^{\infty}k(-1, t)dt \leq c_2\int_0^{\infty}(1 \lor t)^{-\alpha -1}dt < \infty,
    \\
    &\int_{0}^{\infty}t^\rho k(1, -t)dt, \int_{0}^{\infty}t^\rho k(-1, t)dt \leq c_2\int_0^{\infty}t^\rho(1 \lor t)^{-\alpha -1}dt < \infty.
\end{align*}
\end{proof}
From 
the monotone convergence theorem, we obtain the following result.
\begin{corollary}
\label{collary: zbieznosc stalych}
If $\alpha \in (0, 2)$, $\rho \in (-1, \alpha)$, $r \in (0, \infty)$, $k, k_\varepsilon \in K_\alpha$ for small $\varepsilon > 0$, and $k_\varepsilon \nearrow k$ as $\varepsilon \rightarrow 0^+$, then
\[
\lim_{\varepsilon \rightarrow 0^+}C_{+}(k_\varepsilon, r, \rho) = C_{+}(k, r, \rho) \quad\mbox{ and } \quad 
\lim_{\varepsilon \rightarrow 0^+}C_{-}(k_\varepsilon, r, \rho) = C_{-}(k, r, \rho).
\]
\end{corollary}
We denote
    $\mu_\alpha(x, y) := (|x| \lor |y|)^{-\alpha-1}$, $x,y\in \mathbb R$.
\begin{definition} We let
    $K_\alpha^0 := \{ k \in K_\alpha : k \lesssim \mu_\alpha\}$.
\end{definition}

 Observe that $K_\alpha^{0}$ can be characterized as the set of all kernels $k \in K_\alpha$ such that for some constant $c > 0$, we have $k(1, t), k(-1, -t) \leq c$ for all $t \in (0, 1]$.

\begin{lemma}
\label{lemma: aproksymacja k}
For $\alpha \in (0, 2)$ and $k \in K_\alpha$, there exist kernels $k_\varepsilon \in K_\alpha^0$, with $\varepsilon \in (0, 1)$, such that $k_\varepsilon(x, y) \nearrow k(x, y)$ as $\varepsilon \rightarrow 0^+$ for all $x, y \in \R, x \neq y$.
\end{lemma}

\begin{proof}
Let $\varepsilon \in (0, 1)$ and define
\begin{align*}
    U_\varepsilon := \R^2 \setminus \Bigl\{(x, y) \in \R^2 : xy \neq 0,\ 1 - \varepsilon < \frac{x}{y} < \frac{1}{1 - \varepsilon} \Bigr\}.
\end{align*}
Note that $(x, y) \in U_\varepsilon$ if and only if $(y, x) \in U_\varepsilon$ and for every $r > 0$, $(rx, ry) \in U_\varepsilon$ if and only if $(x, y) \in U_\varepsilon$. Let $k_\varepsilon := k\mathds{1}_{U_\varepsilon}$. We will show that $k_\varepsilon \in K_\alpha^0$. The symmetry and homogeneity of $k_\varepsilon$ follow from the corresponding properties of $U_\varepsilon$. Additionally, $k_\varepsilon \leq k \lesssim \Mass$, hence $k_\varepsilon \in K_\alpha$. Furthermore, for $t \in (0, 1]$, we have
\[
k_\varepsilon(1, t),\ k_\varepsilon(-1, -t) \leq c\Mass(1, t)\mathds{1}_{U_\varepsilon}(1, t) \leq c\varepsilon^{-\alpha - 1},
\]
so $k_\varepsilon \in K_\alpha^0$. The pointwise convergence follows since $\bigcup_{\varepsilon \in (0, 1)} U_\varepsilon = \R^2\setminus \{(x, x): x\in \R\}$.

\end{proof}
\ifthesis
\begin{center}
\begin{figure}[H]
\begin{tikzpicture}[scale=3.0]
    \draw[->] (-1,0) -- (1,0) node[right] {$x$};
    \draw[->] (0,-1) -- (0,1) node[above] {$y$};
     
     \begin{scope}
        \clip (0,0) -- (1, 0.6) -- (1, 1) -- (0.6, 1)-- cycle;
        \fill[red!10] (-1.2,-1.2) rectangle (1.2,1.2);
    \end{scope}

    \begin{scope}
        \clip (0,0) -- (-0.6, -1) -- (-1, -1) -- (-1, -0.6) -- cycle;
        \fill[red!10] (-1.2,-1.2) rectangle (1.2,1.2);
    \end{scope}

    \begin{scope}
        \clip (0,0) -- (1, 0.8) -- (1, 1) -- (0.8, 1)-- cycle;
        \fill[red!30] (-1.2,-1.2) rectangle (1.2,1.2);
    \end{scope}

    \begin{scope}
        \clip (0,0) -- (-0.8, -1) -- (-1, -1) -- (-1, -0.8) -- cycle;
        \fill[red!30] (-1.2,-1.2) rectangle (1.2,1.2);
    \end{scope}

    \begin{scope}
        \clip (0,0) -- (1, 0.9) -- (1, 1) -- (0.9, 1)-- cycle;
        \fill[red!60] (-1.2,-1.2) rectangle (1.2,1.2);
    \end{scope}

    \begin{scope}
        \clip (0,0) -- (-0.9, -1) -- (-1, -1) -- (-1, -0.9) -- cycle;
        \fill[red!60] (-1.2,-1.2) rectangle (1.2,1.2);
    \end{scope}
    
    \draw[black,dotted] plot coordinates {
        (1,1) (-1,-1)
    };

    \node[draw, fill=white, anchor=south east] at (1.1,-0.9) {
        \begin{tabular}{l}
            \textcolor{red!10}{\rule{10pt}{6pt}} $\R^2 \setminus U_{0.4}$ \\
            \textcolor{red!30}{\rule{10pt}{6pt}} $\R^2 \setminus U_{0.2}$ \\
            \textcolor{red!60}{\rule{10pt}{6pt}} $\R^2 \setminus U_{0.1}$
        \end{tabular}
    };

\end{tikzpicture}
    \caption{Complements of $U_\varepsilon$.}
\end{figure}
\end{center}
\fi 

\begin{lemma}
\label{lemma: calkowalnosc h}
If $\alpha \in (0, 2)$, $p \in (1, \infty)$, $\rho \in (-1, \alpha)$, $r \in (0, \infty)$, $k \in K_\alpha^0$, and $u \in L^p(|x|^{-\alpha}dx)$, then
\begin{align*}
    \int_\R\int_\R\frac{|u(x)|^p}{h_{\rho, r}(x)}|h_{\rho, r}(x) - h_{\rho, r}(y)|k(x, y)\,dy\,dx < \infty.
\end{align*}
\end{lemma}

\begin{proof}
Using the substitution $y = tx$, for $x > 0$, we obtain
\begin{align*}
&\quad \int_{\R}|h_{\rho, r}(x) - h_{\rho, r}(y)|k(x, y)dy
\\
&\leq c_1\int_{0}^{\infty}||x|^\rho - |y|^\rho|(|x| \lor |y|)^{-\alpha - 1}\,dy 
 + c_1\int_{-\infty}^{0}||x|^\rho - r|y|^\rho|(|x| \lor |y|)^{-\alpha - 1}\,dy
\\
&= c_1|x|^{\rho - \alpha} \left( \int_{0}^{\infty}|1 - |t|^\rho|(1 \lor |t|)^{-\alpha - 1}\,dt + \int_{-\infty}^{0}|1 - r|t|^\rho|(1 \lor |t|)^{-\alpha - 1}\,dt \right)
\\
&\leq c_2|x|^{\rho - \alpha}.
\end{align*}
An analogous estimate holds for $x < 0$. Hence,
\begin{align*}
    \int_\R\int_\R\frac{|u(x)|^p}{h_{\rho, r}(x)}|h_{\rho, r}(x) - h_{\rho, r}(y)|k(x, y)dydx
    &\leq c_3\int_\R\frac{|u(x)|^p}{|x|^\rho}|x|^{\rho - \alpha}dx 
    \\
    &= c_3\int_\R\frac{|u(x)|^p}{|x|^\alpha}dx < \infty.
\end{align*}
\end{proof}

\begin{lemma}
\label{lemma: rownosc Kw}
For $\alpha \in (0, 2)$, $\rho \in (-1, \alpha)$, $r \in (0, \infty)$, and $k \in K_\alpha^0$,
\begin{align}
\label{equation: rownosc Kw}
    \int_{\R}(h_{\rho, r}(x) - h_{\rho, r}(y))k(x, y)\,dy = |x|^{-\alpha}h_{\rho, r}(x)\begin{cases}
        C_{+}(k, r, \rho), \quad &x \geq 0, \\
        C_{-}(k, r, \rho), \quad &x < 0.
    \end{cases}
\end{align}
\end{lemma}

\begin{proof}
Let  $h_+(x) := |x|^{\rho}\mathds{1}_{(0, \infty)}(x)$ and $h_-(x) := |x|^{\rho}\mathds{1}_{(-\infty, 0)}(x)$. Using the substitution $y = tx$ and the homogeneity of $k$, for $x > 0$, we compute
\begin{align*}
    \int_{\R}(h_+(x) - h_+(y))k(x, y)dy 
    &= \int_{0}^{\infty}(|x|^\rho - |y|^\rho)k(x, y)dy + \int_{-\infty}^{0}|x|^\rho k(x, y)dy \\
    &= |x|^{\rho-\alpha} \left( \int_{0}^{\infty}(1 - t^\rho)k(1, t)dt + \int_{0}^{\infty}k(1, -t)dt \right),
\end{align*}
and
\begin{align*}    
    \int_{\R}(h_-(x) - h_-(y))k(x, y)dy 
    &= -\int_{-\infty}^{0}|y|^\rho k(x, y)dy 
    = -|x|^{\rho - \alpha} \int_{0}^{\infty}t^\rho k(1, -t)dt.
\end{align*}
For $x < 0$, we get
\begin{align*}
    \int_{\R}(h_+(x) - h_+(y))k(x, y)\,dy 
    &= -\int_{0}^{\infty}|y|^\rho k(x, y)\,dy 
    = -|x|^{\rho - \alpha} \int_{0}^{\infty}t^\rho k(-1, t)\,dt,
\end{align*}
and
\begin{align*}
    \int_{\R}(h_-(x) - h_-(y))k(x, y)\,dy 
    &= \int_{-\infty}^{0}(|x|^\rho - |y|^\rho)k(x, y)\,dy + \int_{0}^{\infty}|x|^\rho k(x, y)\,dy \\
    &= |x|^{\rho-\alpha} \left( \int_{0}^{\infty}(1 - t^\rho)k(-1, -t)\,dt + \int_{0}^{\infty}k(-1, t)\,dt \right).
\end{align*}
Note that
\[
\int_{1}^{\infty}(1 - t^\rho)k(1, t)\,dt 
= \int_{0}^{1}t^{-2}(1 - t^{-\rho})k(1, t^{-1})\,dt 
= \int_{0}^{1}(t^{\alpha - 1} - t^{\alpha - 1- \rho})k(1, t)\,dt,
\]
so 
\begin{align*}
\int_{0}^{\infty}(1 - t^\rho)k(1, t)\,dt 
&= \int_{0}^{1}(1 - t^\rho + t^{\alpha - 1} - t^{\alpha - 1- \rho})k(1, t)dt \\
&= \int_{0}^{1}(1 - t^\rho)(1 - t^{\alpha - 1- \rho})k(1, t)dt.
\end{align*}
Similarly,
\[
\int_{0}^{\infty}(1 - t^\rho)k(-1, -t)dt = \int_{0}^{1}(1 - t^\rho)(1 - t^{\alpha - 1- \rho})k(-1, -t)dt.
\]

Since $h_{\rho, r} = h_+ + rh_-$, by linearity of the left-hand side of \eqref{equation: rownosc Kw} with respect to $h_{\rho, r}$, the result follows.
\end{proof}

Equality \cref{equation: rownosc Kw} is an important relation between $C_+(k, r, \rho)$, $C_-(k, r, \rho)$, and $h_{\rho, r}$. It may be helpful to note that the argument $\rho$ is the exponent in the power function $h_{\rho, r}$. 

We are now in a position to prove Theorem \cref{theorem: Hardy Identity}. The proof may be regarded as an adaptation of Doob conditioning to the  Sobolev--Bregman forms. 
Let us note that although Doob conditioning, or \( h \)-conditioning, simply relies on substituting \( u \) by \( v h \), where \( v := u/h \) and \( h > 0 \), it is a surprisingly powerful trick.
\subsection*{Proof of Theorem \cref{theorem: Hardy Identity}}
For real numbers $a_1, a_2 \in \R$ and $b_1, b_2 > 0$, we check that
\begin{align}
\label{equation: rownosc liczbowa}
    F_p(a_1, a_2) = F_p \left( \frac{a_1}{b_1}, \frac{a_2}{b_2}\right)b_{1}^{p-1}b_{2}+\left|a_{2}\right|^{p}\cdot\frac{b_{2}^{p-1}-b_{1}^{p-1}}{b_{2}^{p-1}}+\left(p-1\right)\left|a_{1}\right|^{p}\cdot\frac{b_{1}-b_{2}}{b_{1}}.
\end{align}
The identity generalizes Bogdan and Komorowski \cite[(5.20)]{MR3238505} and is inspired by \cite[(2.3) and the proof of Lemma 2.1]{MR4720167} and Bogdan, Dyda, and Luks \cite[Lemma 13]{MR3251822}. 
Applying the identity with $a_1 = u(x)$, $a_2 = u(y)$, $b_1 = h_{p, \beta, s}(x)$, and $b_2 = h_{p, \beta, s}(y)$, for $k \in K_\alpha^0$, we obtain
\begin{align}
\label{equation: rownosc w dowodzie tozsamosci}
    \E_{k, p}[u] = \int_{\R} \int_{\R} \left( g(x, y) + f_1(x, y) + f_2(y, x) \right) k(x, y)\,dx\,dy,
\end{align}
where
\begin{align*}
    g(x, y) &:= \frac{1}{p} F_p\left( \frac{u(x)}{h_{p, \beta, s}(x)}, \frac{u(y)}{h_{p, \beta, s}(y)} \right) h_{p, \beta, s}(x)^{p-1} h_{p, \beta, s}(y), 
\\
f_1(x, y) &:= \frac{1}{q} \frac{|u(x)|^p}{h_{p, \beta, s}(x)}(h_{p, \beta, s}(x) - h_{p, \beta, s}(y)),
\\
f_2(x, y) &:= \frac{1}{p} \frac{|u(x)|^p}{h_{p, \beta, s}(x)^{p-1}}(h_{p, \beta, s}(x)^{p-1} - h_{p, \beta, s}(y)^{p-1}) = \frac{1}{p} \frac{|u(x)|^p}{h_{q, \beta, s}(x)}(h_{q, \beta, s}(x) - h_{q, \beta, s}(y)).
\end{align*}
From Lemma \cref{lemma: calkowalnosc h}, we know that the functions $f_1k$ and $f_2k$ are integrable  with respect to the Lebesgue measure on $\R^2$. The function $gk$ is nonnegative. Therefore, the right-hand side of~\eqref{equation: rownosc w dowodzie tozsamosci} can be split into three separate integrals. Using the symmetry of $k$, Fubini's theorem, and Lemma~\ref{lemma: rownosc Kw}, we obtain:
\begin{align*}
    &\int_{\R} \int_{\R} f_1(x, y)k(x, y)\,dx\,dy + \int_{\R} \int_{\R} f_2(y, x)k(x, y)\,dx\,dy 
    \\
    = &\int_{\R} \int_{\R} f_1(x, y)k(x, y)\,dx\,dy + \int_{\R} \int_{\R} f_2(x, y)k(x, y)\,dx\,dy
    \\
    = &\left(\frac{1}{q}C_+\left(k, s^{\frac{1}{p}}, \frac{\beta}{p}\right) + \frac{1}{p}C_+\left(k, s^{\frac{1}{q}}, \frac{\beta}{q}\right)\right)\int_{0}^{\infty}\frac{|u(x)|^p}{|x|^{\alpha}}dx
    \\
    + &\left(\frac{1}{q}C_-\left(k, s^{\frac{1}{p}}, \frac{\beta}{p}\right) + \frac{1}{p}C_-\left(k, s^{\frac{1}{q}}, \frac{\beta}{q}\right)\right)\int_{-\infty}^{0}\frac{|u(x)|^p}{|x|^{\alpha}}dx
    \\
    = &\int_{\R} |u(x)|^p w_{k, p, s, \beta}(x)dx.
\end{align*}
Thus, the identity \cref{equation: ground state representation} is proved for every $k \in K_\alpha^0$. Next suppose $k \in K_\alpha$. By Lemma~\ref{lemma: aproksymacja k}, there is a family $\{k_\varepsilon\}_{\varepsilon > 0} \subset K_\alpha^0$ such that $k_\varepsilon \nearrow k$ pointwise as $\varepsilon \searrow 0$. Then, by Corollary~\ref{collary: zbieznosc stalych} and the monotone convergence theorem, the ground state representation \cref{equation: ground state representation} also holds for $k$.
The above explains the structure of constants in the definition of $w_{k,p,s,\beta}$, see \eqref{e.dD} and \eqref{e.dV}.
We also note the similarity to \cite[(13)]{MR4372148}.
\qed
\subsection{Criticality}

For $\alpha \in (0, 2)$, $p \in (1, \infty)$, $x,\ y > 0$, $n = 2,3,...$, and functions $u$, let
    \begin{align*}
    N_{\alpha, p}(x, y) &:= x^{\frac{\alpha-1
    }{q}} y^{\frac{\alpha-1}{p}} |x - y|^{-\alpha-1}
    \end{align*}
    and
    \begin{align*}
    \RR_{\alpha, p}(u) &:=  \int_{0}^{\infty}\int_{0}^{\infty}\left(u(x)-u(y)\right)^{2}N_{\alpha, p}(x, y)\,dx\,dy.
\end{align*}
We denote
\begin{align*}
    v_n(x) &:=  \begin{cases}
            nx-1 & \textrm{if $\frac{1}{n} \leq x < \frac{2}{n}$,}\\
            1 & \textrm{if $\frac{2}{n}\leq x < n$,}\\
            2 - x/n & \textrm{if $n \leq x < 2n$,}\\
            0 & \textrm{otherwise.}
            \end{cases}
\end{align*}

\begin{figure}[H]
    \centering
\begin{tikzpicture}
\begin{axis}[
    axis lines = middle,
    xlabel = {$x$},
    ylabel = {}, 
    title = {}, 
    xmin = 0, xmax = 7,
    ymin = -0.2, ymax = 1.2,
    xtick = {0.333, 0.666, 3, 6},
    xticklabels = {$\frac{1}{n}$, $\frac{2}{n}$, $n$, $2n$},
    ytick = {0,1},
    yticklabels = {0,1},
    grid = none,
    width=14cm,
    height=6cm,
    tick style = {thick},
    every axis plot post/.append style={ultra thick, black},
    clip mode=individual,
    axis line style = {thick}
]

\addplot[domain=0:0.333, black, thick] {0};

\addplot[domain=0.333:0.666, black, thick] {3*x - 1};

\addplot[domain=0.666:3, black, thick] {1};

\addplot[domain=3:6, black, thick] {2 - x/3};

\addplot[domain=6:7, black, thick] {0};

\draw[dotted, black!70] (0.666, 0) -- (0.666, 1); 
\draw[dotted, black!70] (3, 0) -- (3, 1);         

\end{axis}
\end{tikzpicture}
    \caption{Plot of $v_n(x)$}
    \label{fig:placeholder}
\end{figure}
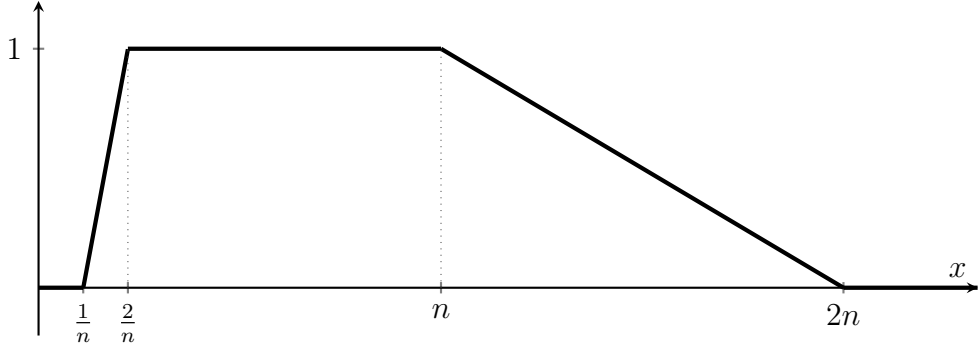

\begin{lemma}
\label{lemma: ograniczonosc} For $\alpha \in (0, 2)$ and $p \in (1, \infty)$, the sequence $\RR_{\alpha, p}(v_n)$, $n = 2,3...$ is bounded.
\end{lemma}

\begin{proof} Observe that
    \begin{align*}
        \RR_{\alpha, p}(v_n) &= \left(
         \int_{\frac{n}{2}}^{3n}\int_{\frac{n}{2}}^{3n}
         +\int_{0}^{\frac{3}{n}}\int_{0}^{\frac{3}{n}}
         +\int_{0}^{2n}\int_{3n}^{\infty} + \int_{3n}^{\infty}\int_{0}^{2n} \right.
         \\
         &+ \left. \int_{0}^{\frac{2}{n}}\int_{\frac{3}{n}}^{n} + \int_{\frac{3}{n}}^{n}\int_{0}^{\frac{2}{n}}
         +\int_{0}^{\frac{n}{2}}\int_{n}^{3n} + \int_{n}^{3n}\int_{0}^{\frac{n}{2}}
         \right) 
         \\
        & \quad \ \left(v_n(x)-v_n(y)\right)^{2}N_{\alpha, p}(x, y)\,dx\,dy
        \\
        &=: I_1 + I_2 + I_3 + I_3' + I_4 + I_4' + I_5 + I_5'.
    \end{align*}
We will show the boundedness of each of these integrals. Let us define $a := \frac{\alpha-1}{q}$, $b := \frac{\alpha-1}{p}$. Note that $a + b = \alpha - 1$. We have
\begin{align*}
    I_1 &\lesssim n^{\alpha - 3}\int_{\frac{n}{2}}^{3n}\int_{\frac{n}{2}}^{3n}|x - y|^{1-\alpha}\,dx\,dy \lesssim n^{\alpha - 3}\int_{\frac{n}{2}}^{3n}\int_{0}^{3n}t^{1-\alpha}dtdy \lesssim n^{\alpha - 3} \cdot n^{2-\alpha} \cdot n = 1,
    \\
    I_2 &\lesssim n^2 \int_{0}^{\frac{3}{n}}\int_{0}^{\frac{3}{n}}|x - y|^{1-\alpha}x^{a}y^{b}\,dx\,dy 
    = 
    n^{\alpha-1-a-b}\int_0^3\int_0^3|x-y|^{1-\alpha}x^ay^b\,dx\,dy
    \lesssim 1,
    \\
    I_3 &\lesssim \int_{0}^{2n}\int_{3n}^{\infty}|x - y|^{-\alpha -1}x^a y^b \,dx\,dy \lesssim \int_{0}^{2n}\int_{3n}^{\infty}x^{a-\alpha-1}y^b\,dx\,dy \lesssim n^{a-\alpha} \cdot n^{b+1} = 1,
    \\
    I_4 &\lesssim \int_{0}^{\frac{2}{n}}\int_{\frac{3}{n}}^{n}|x - y|^{-1 - \alpha}x^a y^b \,dx\,dy = n^{\alpha - a}\int_{0}^{\frac{2}{n}}\int_{3}^{n^2}|t-yn|^{-\alpha -1}t^a y^b dtdy 
    \\
    &\lesssim n^{\alpha-a} \int_{0}^{\frac{2}{n}}y^bdy \lesssim n^{\alpha - a} \cdot n^{-b - 1} = 1,
\end{align*}
and
\begin{align*}
    I_5 &\lesssim \int_{0}^{\frac{n}{2}}\int_{n}^{3n}|x - y|^{-\alpha -1}x^a y^b \,dx\,dy \lesssim n^{-\alpha -1} \cdot n^{a} \cdot n \int_0^\frac{n}{2}y^bdy \lesssim n^{a-\alpha} \cdot n^{b+1} = 1.
\end{align*}
The boundedness of $I_3'$, $I_4'$, and $I_5'$ follows from that of $I_3$, $I_4$, and $I_5$, respectively, but with $p$ and $q$ interchanged.
\end{proof}

\begin{lemma}
\label{lemma: Hilbert}
Let
          \begin{align*}
            \mu(dx) &:= (1\wedge |x|^{-\alpha -1})\,dx,
            \\
            \|f\|_{\alpha, p} &:= \sqrt{\|f\|_{L^2((0,\infty),\mu)}^2 + \RR_{\alpha, p}(f) },
            \\
            X_{\alpha, p} &:= \{ f \in L^2((0,\infty),\mu) : \|f\|_{\alpha, p} <\infty\}.
          \end{align*}
          Then $\mathbb{H}_{\alpha, p} := (X_{\alpha, p}, \|\cdot\|_{\alpha, p})$ is a Hilbert space.
\end{lemma}

\begin{proof} Observe that $\|\cdot\|_{\alpha, p}$ is induced by the inner product of the form
\begin{align*}
    (f, g) := (f, g)_{L^2((0, \infty), \mu)} + \int_{0}^{\infty} \int_{0}^{\infty}\left(f(x)-f(y)\right)\left(g(x)-g(y)\right)N_{\alpha, p}(x, y)\,dx\,dy.
\end{align*}
Therefore, it is enough to prove the completeness of $\mathbb{H}_{\alpha, p}$. To this end, let us consider a Cauchy sequence $(f_n)_{n \in \mathbb{N}}$ in $\mathbb{H}_{\alpha, p}$. Of course, the sequence is Cauchy in $L^2((0,\infty),\mu)$, too; hence it converges strongly to some $g \in L^2((0,\infty),\mu)$. Furthermore, it has a subsequence $(f_{n_j})_{j \in \mathbb{N}}$ converging to $g$ almost everywhere, because $\mu$ is $\sigma$-finite, in fact, finite. Observe that $\RR_{\alpha, p}(f_n) = \|F_n\|^2_{L^2((0, \infty)^2, \eta)}$, where $F_n(x, y) := f_n(x)-f_n(y)$ and $\eta(dx, dy) := N_{\alpha, p}(x, y) \,dx\,dy$, and $(F_n)_{n \in \mathbb{N}}$ is Cauchy in $L^2((0, \infty)^2, \eta)$; hence it converges strongly to some $G \in L^2((0, \infty)^2, \eta)$. By Fubini's theorem,
\begin{align*}
    0 = \liminf_{j \rightarrow \infty}\| F_{n_j}  - G\|_{L^2((0,\infty)^2, \eta)}^2 \geq \int_{(0,\infty)^2}\liminf_{j \rightarrow \infty}(F_{n_j} - G)^2d\eta \geq 0,
\end{align*}
which implies
\begin{align*}
    G(x, y) = \lim_{j \rightarrow \infty}F_{n_j}(x, y) = g(x) - g(y), \quad \text{a.e.}
\end{align*}
Therefore $(f_n)_{n \in \mathbb{N}}$ converges strongly to $g$ in $\mathbb{H}_{\alpha, p}$.
\end{proof}

\begin{lemma}
\label{lemma: critical weight of remainder}
    Suppose that for some function $\phi \geq 0$, 
          \[
           \RR_{\alpha, p}(v) \geq \int_0^\infty v^2(x) \phi(x)\,dx , \quad v\in C_c(0,\infty).
          \]
          Then $\phi=0$ a.e.
\end{lemma}
\begin{proof}
          By scaling, it is enough to show that $\phi=0$ on the interval $(1,2)$.
          Let us consider the functions $v_n\in C_c(0,\infty)$ from Lemma \cref{lemma: ograniczonosc} and the Hilbert space $\mathbb{H}_{\alpha, p}$ from Lemma \cref{lemma: Hilbert}.
          By Lemma \ref{lemma: ograniczonosc}, the sequence $(v_n)_{n \in \mathbb{N}}$ is bounded in $\mathbb{H}_{\alpha, p}$. Since $\mathbb{H}_{\alpha, p}$ 
          is necessarily a reflexive space,
          by Banach--Alaoglu and Eberlein--\v{S}mulian theorems,
          there exists a subsequence $(v_{n_{k}})_{k\in\mathbb{N}}$ weakly convergent to some $v$.
          Since $(v_{n_{k}})_{k\in\mathbb{N}}$ converges to $1$ in $L^2((0,\infty), \mu)$,
          it follows that $v = 1$.
          By Mazur's lemma, there are functions $f_k$ such that each $f_k$ is a convex combination
          of $v_{n_1}, \ldots, v_{n_k}$ and  $\| f_k - v\|_{\alpha, p} \xrightarrow[]{k \rightarrow\infty} 0$. Since each $v_n=1$ on $(1,2)$,
          each $f_k=1$ on $(1,2)$, and so
          \[
          \int_1^2 \phi(x)\,dx =
          \int_1^2 f_k^2(x)\phi(x)\,dx \leq
          \RR_{\alpha, p}(f_k) = \RR_{\alpha, p}(f_k - v) \xrightarrow[]{k\rightarrow\infty} 0.
          \]
          Therefore $\phi=0$ a.e. on $(1,2)$.
\end{proof}

\begin{lemma}
\label{lemma: comparability}
For every $p \in (1, \infty)$, there is a constant $c_p$ such that
\begin{align*}
    F_p\left(a^{\langle \frac{2}{p} \rangle}, b^{\langle \frac{2}{p} \rangle}\right) \leq c_p(a - b)^2, \quad a, b \in \R.
\end{align*}
\end{lemma}

Lemma \cref{lemma: comparability} is proven, e.g., in \cite[Lemma 2.3.]{MR4589708}

\begin{lemma}
\label{lemma: pfrom R comparison}
For $\alpha \in (0, 2)$, conjugate exponents $p$ and $q$, $s \in (0, \infty)$, and $k \in K_\alpha$, there is a constant $c_{k, p, s}$ such that for every symmetric function $v$,
\begin{align*}
    \int_{\R}\int_{\R}F_p \left(v(x)^{\langle \frac{2}{p} \rangle}, v(y)^{\langle \frac{2}{p} \rangle}\right) h_{q, \alpha-1, s}(x)h_{p, \alpha-1, s}(y) k(x, y)\,dx\,dy \leq c_{k, p, s}\RR_{\alpha, p}(v).
\end{align*}
\end{lemma}

\begin{proof}Let $\beta = \alpha - 1$. Since $k \lesssim \Mass$, by Lemma \cref{lemma: comparability}, we have
\begin{align*}
    &\int_{\R}\int_{\R}F_p \left(v(x)^{\langle \frac{2}{p} \rangle}, v(y)^{\langle \frac{2}{p} \rangle}\right) h_{q, \beta, s}(x)h_{p, \beta, s}(y) k(x, y)\,dx\,dy
    \\
    \leq &c'\int_{\R}\int_{\R}(v(x) - v(y))^2 h_{q, \beta, s}(x)h_{p, \beta, s}(y) \Mass(x, y)\,dx\,dy
    \\
    = &\frac{c'}{2}\int_{\R}\int_{\R}(v(x) - v(y))^2 (h_{q, \beta, s}(x)h_{p, \beta, s}(y) + h_{q, \beta, s}(y)h_{p, \beta, s}(x)) \Mass(x, y)\,dx\,dy =: I_1.
\end{align*}
Observe that $h_{p, \beta, s} \lesssim h_{p, \beta, 1}$ and $h_{q, \beta, s} \lesssim h_{q, \beta, 1}$. Let 
\begin{align*}
    l(x, y) := (h_{q, \beta, 1}(x)h_{p, \beta, 1}(y) + h_{q, \beta, 1}(y)h_{p, \beta, 1}(x)) \Mass(x, y).
\end{align*}
We get
\begin{align*}
        I_1 \leq & c''\int_{\R}\int_{\R}(v(x) - v(y))^2l(x, y)\,dx\,dy
        \\
        = & c''\left(\int_{0}^{\infty}\int_{0}^{\infty} + 2\int_{0}^{\infty}\int_{-\infty}^{0}+ \int_{-\infty}^{0}\int_{-\infty}^{0}  \right)(v(x) - v(y))^2 l(x, y)\,dx\,dy =: I_2+2I_3+I_4.
\end{align*}
Observe that $l(x, y) = l(-x, -y)$ for $x, y > 0$ hence $I_4=I_2$. Moreover $l(-x, y) \leq l(x, y)$ for $x, y > 0$ thus, using the substitution $x = -t$, we see that $I_3\le I_2$.
Therefore,
\begin{align*}
    I_1 \leq 4 c''\int_{0}^{\infty}\int_{0}^{\infty}(v(x) - v(y))^2 l(x, y)\,dx\,dy = 8c'' \RR_{\alpha, p}(v).
\end{align*}
\end{proof}

We are now in a position to prove Theorem \cref{theorem: Hardy Inequality}.
\subsection*{Proof of Theorem \cref{theorem: Hardy Inequality}}
Suppose that $W \geq W_{k, p, s}$ a.e. and $W$ is a $(k, p)$-Hardy weight, i.e.,
\begin{align}
\label{equation: supouse inequality}
    \E_{k, p}[u] \geq \int_{\R}|u(x)|^pW(x)dx, \quad u \in C_c(\Ro).
\end{align}
For $v \in C_c(0, \infty)$, let us define \begin{align}\label{e.da}
    u(x) := h_{p, {\alpha - 1}, s}(x)v(|x|)^{\langle\frac{2}{p}\rangle}. 
\end{align}
By Theorem \cref{theorem: Hardy Identity} the left hand side of \cref{equation: supouse inequality} equals to
\begin{align*}
    &\frac{1}{p}\int_{\R}\int_{\R}F_p \left( \frac{u(x)}{h_{p, \alpha - 1, s}(x)}, \frac{u(y)}{h_{p, \alpha - 1, s}(y)}\right) h_{q, \alpha - 1, s}(x)h_{p, \alpha - 1, s}(y) k(x, y)\,dx\,dy\\ + &\int_{\R}|u(x)|^pW_{k, p, s}(x)dx
    \\
    = &\frac{1}{p}\int_{\R}\int_{\R}F_p \left( v(|x|)^{\langle\frac{2}{p}\rangle}, v(|y|)^{\langle\frac{2}{p}\rangle}\right) h_{q, \alpha - 1, s}(x)h_{p, \alpha - 1, s}(y) k(x, y)\,dx\,dy\\ + &\int_{\R}v(|x|)^2 h_{\alpha-1, s}(x) W_{k, p, s}(x)dx. 
\end{align*}
Thus, by Lemma~\cref{lemma: pfrom R comparison}, for some constant \( c > 0 \), we have that
\[
    \RR_{\alpha, p}(v) \geq \int_{\R} v(|x|)^2 \cdot c\cdot h_{\alpha - 1, s}(x)\, (W(x) - W_{k, p, s}(x))\, dx, \quad v \in C_c(0, \infty).
\]
Therefore, by Lemma~\cref{lemma: critical weight of remainder}, it follows that \( c\, h_{\alpha-1, s}(W - W_{k, p, s}) = 0 \) a.e., which implies that \( W = W_{k, p, s} \) a.e.
\qed
\subsection{Balancing and improving critical Hardy inequalities}
\subsection*{Proof of Theorem \texorpdfstring{\cref{theorem: sign}}{Theorem sign}}
\begin{proof}
Let us recall that
\begin{align*}
    W_{k, p, s}(x) := |x|^{-\alpha}\begin{cases}
        D_+(k, p, s, \alpha - 1), &\quad x > 0,
        \\
        D_-(k, p, s, \alpha - 1), &\quad x < 0.
    \end{cases}
\end{align*}

Then (a) follows from the definition \cref{definition: D} since $D_+(k, p, 1, 0) = D_{-}(k, p, 1, 0) = 0$.

Next we prove (b). Observe that $D_+(k, p, s, \alpha-1) = \gamma_+(k, \frac{\alpha-1}{p}), D_{-}(k, p, s, \alpha-1) = \gamma_-(k, \frac{\alpha-1}{p})$. Since $\alpha \neq 1$ and we excluded the case $k = 0$ a.e., at least one of the values $\gamma_+(k, \frac{\alpha-1}{p}), \gamma_-(k, \frac{\alpha-1}{p})$ has to be positive---see the Definition~\ref{definition: gamma C}.

Finally, we prove (c) as follows. Observe that $D_+(k, p, s, \alpha - 1)$  is strictly decreasing and continuous in $s$ with the limit $-\infty$ at $\infty$. Similarly $D_-(k, p, s, \alpha - 1)$ as the function of $s$ is strictly increasing and continuous with the limit $-\infty$ at $0$. Thus there exists exactly one $s_0 \in (0, \infty)$, such that $D_+(k, p, s_0, \alpha - 1) = D_-(k, p, s_0, \alpha - 1) =: D$. Since $W_{k, p, s_0}(x) = D|x|^{-\alpha}$ is a~critical $(k, p)$-Hardy weight, it follows that $D \geq 0$. Suppose that $D = 0$. By the definition \cref{definition: D} of $D_+$, we have
\begin{align}
\begin{split}
\label{equation: rownosc s0}
    &\frac{1}{q}s_0^{\frac{1}{p}}\int_{0}^{\infty}t^{\frac{\alpha-1}{p}}k\left(1,-t\right)dt+\frac{1}{p}s_0^{\frac{1}{q}}\int_{0}^{\infty}t^{\frac{\alpha-1}{q}}k\left(1,-t\right)dt 
    \\
    =&\gamma_+\left(k, \frac{\alpha-1}{p}\right)+\int_{0}^{\infty}k\left(1,-t\right)dt.
\end{split}
\end{align}
By substitution $t = \frac{1}{u}$,
\begin{align*}
    \int_{0}^{\infty}k\left(-1,t\right)dt = \int_{0}^{\infty}t^{\alpha-1}k\left(1,-t\right)dt
\end{align*}
and
\begin{align}\label{eq:k(-1,t)}
    \int_{0}^{\infty}t^{\frac{\alpha-1}{r_1}}k\left(-1,t\right)dt = \int_{0}^{\infty}t^{\frac{\alpha-1}{r_2}}k\left(1,-t\right)dt
\end{align}
 whenever $r_1$ and $r_2$ are conjugate exponents. By the equality \cref{equation: rownosc s0}, we get
\begin{align*}
    0 = &s_0D_{-}(k, p, s_0, \alpha-1)
    \\
    = &s_0\gamma_-\left(k, \frac{\alpha-1}{p}\right) + s_0\int_0^{\infty}k(-1, t)dt 
    \\
    - &\frac{1}{p}s_0^{\frac{1}{p}}\int_{0}^{\infty}t^{\frac{\alpha-1}{q}}k\left(-1,t\right)dt - \frac{1}{q}s_0^{\frac{1}{q}}\int_{0}^{\infty}t^{\frac{\alpha-1}{p}}k\left(-1,t\right)dt
    \\
    = &s_0\gamma_-\left(k, \frac{\alpha-1}{p}\right) + s_0\int_0^{\infty}t^{\alpha-1}k(1, -t)dt 
    \\
    + &\frac{1}{q}s_0^{\frac{1}{p}}\int_{0}^{\infty}t^{\frac{\alpha-1}{p}}k\left(1,-t\right)dt + \frac{1}{p}s_0^{\frac{1}{q}}\int_{0}^{\infty}t^{\frac{\alpha-1}{q}}k\left(1,-t\right)dt
    \\
    - &s_0^{\frac{1}{p}}\int_{0}^{\infty}t^{\frac{\alpha-1}{p}}k\left(1,-t\right)dt - s_0^{\frac{1}{q}}\int_{0}^{\infty}t^{\frac{\alpha-1}{q}}k\left(1,-t\right)dt
    \\
    = &s_0\gamma_-\left(k, \frac{\alpha-1}{p}\right) + \gamma_+\left(k, \frac{\alpha-1}{p}\right) 
    \\
    + &\int_{0}^{\infty}k\left(1,-t\right)dt+s_0\int_0^{\infty}t^{\alpha-1}k(1, -t)dt 
    \\
    - &s_0^{\frac{1}{p}}\int_{0}^{\infty}t^{\frac{\alpha-1}{p}}k\left(1,-t\right)dt - s_0^{\frac{1}{q}}\int_{0}^{\infty}t^{\frac{\alpha-1}{q}}k\left(1,-t\right)dt
    \\
    = &s_0\gamma_-\left(k, \frac{\alpha-1}{p}\right) + \gamma_+\left(k, \frac{\alpha-1}{p}\right) + \int_{0}^{\infty}\left(1 - (s_0 t^{\alpha-1})^{\frac{1}{p}}\right)\left(1 - (s_0 t^{\alpha-1})^{\frac{1}{q}}\right)k(1, -t)dt.
\end{align*}
All the terms in the last equality are nonnegative, which in particular implies that the last one must be equal to $0$. This happens only when $\alpha = s_0 = 1$, which, however, contradicts  the assumption that $\alpha \neq 1$. Thus, $D > 0$ and so $W_{k, p, s_0} > 0$ a.e. Moreover $D_+$ and $D_-$ are continuous as functions of $s$, hence there is an  open neighborhood $S$ of $s_0$ such that $W_{k,p,s} > 0$ a.e. for all $s \in S$.
\end{proof}

\subsection*{Proof of Proposition \texorpdfstring{\cref{proposition: extra term}}{proposition: extra term}}

\begin{proof} Let $u \in L^p(|x|^{-\alpha}dx)$
  and
\begin{align*}
    S := \sup_{s > 0}f(s),
\end{align*}
where
\begin{align*}
    f(s) := \int_{\R} |u(x)|^p W_{k,p,s}(x)\,dx.
\end{align*}
First suppose that $I_+ = 0$. Then $\mathcal{G}_{k, p}(u) = \left(\frac{l_q}{q} + \frac{l_p}{p}\right)I_-$ and
\begin{align*}
    S &= \lim_{s \rightarrow \infty}D_-(k, p, s, \alpha - 1)I_- = \left(\gamma_{-}\left(k, \frac{\alpha - 1}{p}\right) + \int_{0}^{\infty}k(-1, t)dt\right)I_-
    \\
    &= \left(\gamma_{-}\left(k, \frac{\alpha - 1}{p}\right) + \int_{0}^{\infty}k(-1, t)dt - \left(\frac{l_q}{q} + \frac{l_p}{p}\right) + \left(\frac{l_q}{q} + \frac{l_p}{p}\right)\right)I_- 
    \\
    &= D_-(k, p, 1, \alpha - 1)I_- + \mathcal{G}_{k, p}(u) = f(1) + \mathcal{G}_{k, p}(u),
\end{align*}
where we used \eqref{eq:k(-1,t)} in the last equality.
Similarly we resolve the case $I_- = 0$. Now suppose that $I_+, I_- > 0$. We get,
again using \eqref{eq:k(-1,t)},
\begin{align*}
    f'(s) = \frac{1}{pq}\left(s^{-\frac{1}{q} - 1}l_p + s^{-\frac{1}{p} - 1}l_q
    \right)I_- - \frac{1}{pq}\left(
    s^{\frac{1}{q} - 1}l_q + s^{\frac{1}{p} - 1}l_p\right)I_+,\quad s>0.
\end{align*}
Let
\begin{align*}
    g(s) := spq \cdot f'(s) =   \left(s^{-\frac{1}{q}}l_p + s^{-\frac{1}{p}}l_q
    \right)I_- - \left(
    s^{\frac{1}{q}}l_q + s^{\frac{1}{p}}l_p\right)I_+, \quad s>0.
\end{align*}
Of course, $spq > 0$ implies that $\sign(g(s)) = \sign(f'(s))$ for all $s > 0$. Moreover $g$ is strictly decreasing and $g\left(\frac{I_-}{I_+}\right)=0$. This implies that $S = f\left( \frac{I_-}{I_+}\right)$. We calculate
\begin{align*}
    f\left( \frac{I_-}{I_+}\right) &= \left( \gamma_{+}\left(k, \frac{\alpha - 1}{p}\right) + \int_{0}^{\infty}k(1, -t)dt \right)I_+ - \frac{1}{p}I_{-}^{\frac{1}{q}}I_{+}^{\frac{1}{p}}l_q - \frac{1}{q}I_{-}^{\frac{1}{p}}I_{+}^{\frac{1}{q}}l_p
    \\
    &+\left( \gamma_{-}\left(k, \frac{\alpha - 1}{p}\right) + \int_{0}^{\infty}k(-1, t)dt \right)I_- - \frac{1}{p}I_{+}^{\frac{1}{q}}I_{-}^{\frac{1}{p}}l_p - \frac{1}{q}I_{+}^{\frac{1}{p}}I_{-}^{\frac{1}{q}}l_q
    \\
    &=\left( \gamma_{+}\left(k, \frac{\alpha - 1}{p}\right) + \int_{0}^{\infty}k(1, -t)dt - \frac{1}{p}l_q - \frac{1}{q}l_p \right)I_+
    \\
    &+\left( \gamma_{-}\left(k, \frac{\alpha - 1}{p}\right) + \int_{0}^{\infty}k(-1, t)dt - \frac{1}{p}l_p - \frac{1}{q}l_q \right)I_-
    \\
    &+ l_p\left( \frac{I_+}{q} + \frac{I_-}{p} -I_{-}^{\frac{1}{p}}I_{+}^{\frac{1}{q}} \right) + l_q\left( \frac{I_+}{p} + \frac{I_-}{q} -I_{+}^{\frac{1}{p}}I_{-}^{\frac{1}{q}} \right)
    \\
    &= f(1) + \mathcal{G}_{k, p}(u).
\end{align*}
\end{proof}

\section{Examples and miscellanea}
\label{sec:Example}
In this section we focus on implications of Theorems
 \cref{theorem: Hardy Inequality} and \cref{proposition: extra term} for $p = 2$ and kernels
 \begin{align*}
     \Masss(x, y) := |x - y|^{- \alpha - 1}\begin{cases}
         a_+, \quad x, y \geq 0,
         \\
         a_-, \quad x, y \leq 0,
         \\
         a_{\pm}, \quad \sign(x) \neq \sign(y), |x \lor y| =|x| \lor |y|,
         \\
         a_{\mp}, \quad \sign(x) \neq \sign(y), |x \lor y| \neq |x| \lor |y|,
     \end{cases}
\end{align*}
where $\boldsymbol{a} = (a_+, a_-, a_{\pm}, a_{\mp}) \in [0, \infty)^4$. Of course, $\alpha \in (0, 2)$. Figure~\ref{f.apm} shows the different regions of  $\R^2$, where the coefficients $a_+, a_-, a_\pm, a_\mp$ apply.
\begin{figure}[h]
\centering
Coefficients $a_+, a_-, a_\pm, a_\mp$ in the definition of $\Mass(x,y)$.
\begin{tikzpicture}[scale=3]
  \draw[->] (-1,0) -- (1,0) node[right] {$x$};
  \draw[->] (0,-1) -- (0,1) node[above] {$y$};

  \draw (-1,1) -- (1,-1);

  \node at (0.5,0.5) {$a_+$};
  \node at (-0.5,-0.5) {$a_-$};
  \node at (-0.25,0.5) {$a_\pm$};
  \node at (-0.5,0.25) {$a_\mp$};
  \node at (0.25,-0.5) {$a_\mp$};
  \node at (0.5,-0.25) {$a_\pm$};
\end{tikzpicture}
\caption{}
\label{f.apm}
\end{figure}
For example, the constant $a_\pm$ is used when the arguments $x,y$ of $\Mass(x,y)$ have different signs, and the one with the larger absolute value is positive.

\subsection{Constants}
The constant 
\begin{equation}
\label{equation: gamma def}
\gamma(\alpha, \rho) := \int_{0}^{1}\frac{\left(t^{\rho}-1\right)\left(1-t^{\alpha-1-\rho}\right)}{\left(1-t\right)^{\alpha+1}}dt, \quad \alpha \in (0, 2), \; \rho \in (-1, \alpha),
\end{equation}
appeared in Bogdan and Burdzy \cite[(5.2)]{MR2006232} and was widely used since to describe the boundary behavior of superharmonic functions of nonlocal operators with scaling, see, e.g., Kim, Song, and Vondraček \cite{MR4886451}. Here is a new formula for $\gamma$ in terms of the beta function
\begin{align}
B(x, y) = \frac{\Gamma(x)\Gamma(y)}{\Gamma(x + y)}, \quad x, y > 0.
\end{align}
\begin{lemma}
\label{lemma: gamma laplasjan}
For $\alpha \in (0, 2)$ and $\rho \in (-1, \alpha)$, 
\begin{align}
\label{equation: gamma laplasjan}
-\gamma(\alpha, \rho) = \frac{\sin\left(\pi\left(\frac{\alpha}{2}-\rho\right)\right)}{\sin(\pi\alpha/2)}B\left(\rho+1,\alpha-\rho\right)-\frac{1}{\alpha}.
\end{align}
\end{lemma}

\begin{proof} Let $-\widetilde{\gamma}(\alpha, \rho)$ denote the right-hand side of \cref{equation: gamma laplasjan}. We put
\begin{align*}
    X &:= \{ (\alpha, \rho) \in \R^2 : \alpha \in (0, 2), \rho \in (-1, \alpha)\},
    \\
    X_0 &:= X \setminus
    \bigg( \{(\alpha,\beta) : \beta \in \{0,1,\alpha-1,\alpha-2\},\, \alpha\in(0,2) \} \cup \Big( \{1\}\times (-1,1) \Big) \bigg).
\end{align*}
Observe that $\widetilde{\gamma}$ is continuous on $X$ and the set $X_0$ is dense in $X$. Therefore it is enough to verify that  $\gamma$ and $ \widetilde{\gamma}$ coincide on $X_0$ and that $\gamma$ is continuous on $X$. We first prove the former. In Bogdan and Dyda \cite[(2.2)]{2011-KB-BD-mn}, it is shown that for $\alpha \neq 1$,
\begin{align}
\label{equation: rownosc dla gammy}
-\gamma(\alpha, \rho) &= \frac{\left(\alpha-1-\rho\right)\left(\alpha-\rho-2\right)}{\alpha\left(1-\alpha\right)}B\left(\rho+1, 2-\alpha\right)\\
&+\frac{\rho\left(\rho-1\right)}{\alpha\left(1-\alpha\right)}B\left(\alpha-\rho, 2-\alpha\right)-\frac{1}{\alpha}.
\nonumber
\end{align}
In particular, \cref{equation: rownosc dla gammy} is true on $X_0$. Fix $(\alpha, \rho) \in X_0$. We get
\begin{align*}
    &\frac{B\left(\rho+1,2-\alpha\right)}{B\left(\rho+1,\alpha-\rho\right)} = \frac{\Gamma\left(\rho+1\right)\Gamma\left(2-\alpha\right)\Gamma\left(\alpha+1\right)}{\Gamma\left(3+\rho-\alpha\right)\Gamma\left(\rho+1\right)\Gamma\left(\alpha-\rho\right)} 
    \\
    = &\frac{\alpha\left(\alpha-1\right)}{\left(\rho-\alpha+1\right)\left(\rho-\alpha+2\right)}\frac{\Gamma\left(-\alpha\right)\Gamma\left(1-\left(-\alpha\right)\right)}{\Gamma\left(\alpha-\rho\right)\Gamma\left(1-\left(\alpha-\rho\right)\right)} = \frac{\alpha\left(1-\alpha\right)}{\left(\alpha-1-\rho\right)\left(\alpha-\rho-2\right)}\frac{\sin\left(\pi\left(\alpha-\rho\right)\right)}{\sin\left(\pi\alpha\right)}.
\end{align*}
Similarly,
\begin{align*}
    \frac{B\left(\alpha-\rho,\ 2-\alpha\right)}{B\left(\rho+1,\ \alpha-\rho\right)} = \frac{\alpha\left(\alpha-1\right)}{\rho\left(\rho-1\right)}\frac{\sin\left(\pi \rho\right)}{\sin\left(\pi\alpha\right)}.
\end{align*} 
Thus,
\begin{align*}
    \frac{-\gamma(\alpha, \rho) + \frac{1}{\alpha}}{B\left(\rho+1,\ \alpha-\rho\right)} = \frac{\sin\left(\pi\left(\alpha-\rho\right)\right)-\sin\left(\pi \rho\right)}{\sin\left(\pi\alpha\right)} = \frac{\sin\left(\pi\left(\frac{\alpha}{2}-\rho\right)\right)}{\sin(\pi\alpha/2)}.
\end{align*}
The last equality holds because of the trigonometric identities 
\begin{align*}
    \sin\left(x\right)-\sin\left(y\right)&=2\cos\left(\frac{x+y}{2}\right)\sin\left(\frac{x-y}{2}\right),
    \\
    \sin(x)&=2\cos\left(\frac{x}{2}\right)\sin\left(\frac{x}{2}\right).
\end{align*}
We next  show the continuity of $\gamma$ on $X$. Let 
\begin{align*}
    g_{\alpha, \rho}(x) := \frac{\left(1-x^{\rho}\right)\left(1-x^{\alpha-1-\rho}\right)}{\left(1-x\right)^{\alpha+1}}, \quad (\alpha, \rho) \in X,\; x \in (0, 1).
\end{align*}
 Let $\alpha_1 < \alpha_2$, $\rho_1 < \rho_2$, $[\alpha_1, \alpha_2] \times [\rho_1, \rho_2] \subset X$. We introduce a majorant of $g_{\alpha, \rho}$ for all $(\alpha, \rho) \in [\alpha_1, \alpha_2] \times [\rho_1, \rho_2]$. When $x \in (0, \frac{1}{2})$,
\begin{align*}
    |g_{\alpha, \rho}(x)| \leq \frac{1 + x^\rho + x^{\alpha - 1- \rho} + x^{\alpha - 1}}{(\frac{1}{2})^{\alpha + 1}} \leq 2^{\alpha_2 + 1}(1 + x^{\rho_1} + x^{\alpha_1 - \rho_2 - 1} + x^{\alpha_1 - 1}),
\end{align*}
which is integrable on $(0, \frac{1}{2})$ because $\rho_1, \alpha_1 - 1, \alpha_1 - \rho_2 - 1 > -1$. We next consider $x \in [\frac{1}{2}, 1)$. For $r \in (-1, 2)$, we have $|1 - x^r| \leq |1 - x^{-1}| \lor |1 - x^2| = |1-x|(\frac{1}{|x|} \vee |1+x|) \leq 2(1 - x)$.
  Thus,
\begin{align*}
    |g_{\alpha, \rho}(x)| < \frac{4(1 - x)^2}{(1 - x)^{\alpha + 1}} \leq 4(1 - x)^{1 - \alpha_2},\quad x\in [1/2,1),
\end{align*}
which is integrable on $[\frac{1}{2}, 1)$ because $1 - \alpha_2 > -1$. By the dominated convergence theorem, $\gamma$ is continuous on the rectangle $(\alpha_1, \alpha_2) \times (\rho_1, \rho_2)$.
\end{proof}
Here are a few basic properties of $\gamma$ resulting from \cref{equation: gamma laplasjan}.
\begin{corollary}The following identities hold,
\begin{align*}
    \gamma(\alpha, \rho) &= \gamma(\alpha, \alpha - 1- \rho), && \alpha \in (0, 2),\ \rho \in (-1, \alpha), \\
    \gamma\left( \alpha, \frac{\alpha}{2} \right) &= \frac{1}{\alpha}, && \alpha \in (0, 2), \\
    \gamma(1, \rho) &=
    1 - \pi\rho\cot(\pi\rho), && \rho \in (-1, 1)\setminus\{0\}.
\end{align*}
\end{corollary}
Moreover, it is known that for any $\alpha \in (0, 2)$, the function $\rho \mapsto -\gamma(\alpha, \rho), \rho \in (-1, \alpha)$ achieves its maximum at the point $\rho = \frac{\alpha-1}{2}$,
see \cite[page 632, paragraph between (1.8) and (1.9)]{2011-KB-BD-mn}. We have,
\begin{align*}
    -\gamma \left(\alpha, \frac{\alpha - 1}{2} \right) = B\left(\frac{\alpha + 1}{2}, \frac{\alpha + 1}{2} \right)/\sin(\pi\alpha/2) - \frac{1}{\alpha}, &&\alpha \in (0, 2).
\end{align*}
Instead of \eqref{equation: gamma laplasjan}, we can write
\begin{align*}
    -\gamma(\alpha, \rho) &= \frac{B\left(\frac{\alpha}{2},\ 1-\frac{\alpha}{2}\right)B\left(\rho+1,\alpha-\rho\right)}{B\left(\frac{\alpha}{2}-\rho,\ 1-\frac{\alpha}{2}+\rho\right)}-\frac{1}{\alpha}
    \\
    &=\frac{\Gamma\left(\frac{\alpha}{2}\right)\Gamma\left(1-\frac{\alpha}{2}\right)\Gamma\left(\rho+1\right)\Gamma\left(\alpha-\rho\right)}{\Gamma\left(1-\frac{\alpha}{2}+\rho\right)\Gamma\left(\alpha+1\right)\Gamma\left(\frac{\alpha}{2}-\rho\right)}-\frac{1}{\alpha}, \quad \alpha \in (0, 2),\ \rho \in (-1, \alpha).
\end{align*}
By Corollary \cref{corollary: case p = 2 critical weight}, for every $r \in (0, \infty)$,
\begin{align*}
    M_{\Masss, r}(x) = |x|^{-\alpha}\begin{cases}
        C_+\left(\Masss, r, \frac{\alpha - 1} {2}\right), \quad &x \geq 0,
        \\[1ex]
        C_-\left(\Masss, r, \frac{\alpha - 1} {2}\right), \quad &x < 0
    \end{cases}
\end{align*}
is a critical $\left(\Masss, 2\right)$-Hardy weight.
\begin{proposition}
\label{proposition: constants nu^a}
For $\alpha \in (0, 2)$ and $r \in (0, \infty)$, we have
\begin{align*}
    C_{+}\left(\Masss, r, \frac{\alpha - 1}{2}\right) &= B\left(\frac{\alpha+1}{2},\frac{\alpha+1}{2}\right) \left(\frac{a_+}{\sin(\pi\alpha/2)} - r\frac{a_\pm + a_\mp}{2} \right) 
    \\
    &+ \frac{a_\pm - a_+}{\alpha} + \frac{a_\mp -a_\pm}{\alpha 2^\alpha},
    \\
    C_{-}\left(\Masss, r, \frac{\alpha - 1}{2}\right) &= B\left(\frac{\alpha+1}{2},\frac{\alpha+1}{2}\right) \left(\frac{a_-}{\sin(\pi\alpha/2)} -\frac{1}{r} \frac{a_\pm + a_\mp}{2}\right) 
    \\
    &+ \frac{a_\mp - a_-}{\alpha} + \frac{a_\pm - a_\mp}{\alpha 2^\alpha}.
\end{align*}
\end{proposition}

\begin{proof} Using the Definition~\ref{definition: gamma C}, we calculate
\begin{align*}
    \int_0^\infty k(1, -t)dt = a_\pm \int_0^1 (1 + t)^{-\alpha - 1}dt + a_\mp \int_1^\infty (1 + t)^{-\alpha - 1}dt = \frac{a_\pm}{\alpha} + \frac{a_\mp - a_\pm}{\alpha 2^\alpha},
    \\
    \int_0^\infty k(-1, t)dt = a_\mp \int_0^1 (1 + t)^{-\alpha - 1}dt + a_\pm \int_1^\infty (1 + t)^{-\alpha - 1}dt = \frac{a_\mp}{\alpha} + \frac{a_\pm - a_\mp}{\alpha 2^\alpha}.
\end{align*}
By the change of variables $t = \frac{1}{u}$, we get
\begin{align*}
    \int_{0}^1 \frac{t^{\frac{\alpha - 1}{2}}}{(1 + t)^{\alpha+1}}dt = \int_{1}^\infty \frac{t^{\frac{\alpha - 1}{2}}}{(1 + t)^{\alpha+1}}dt = \frac{1}{2}\int_{0}^\infty \frac{t^{\frac{\alpha - 1}{2}}}{(1 + t)^{\alpha+1}}dt = \frac{1}{2}B\left(\frac{\alpha+1}{2},\frac{\alpha+1}{2}\right),
\end{align*}
thus
\begin{align*}
\int_{0}^\infty t^{\frac{\alpha - 1}{2}}k(1, -t)dt = \int_{0}^\infty t^{\frac{\alpha - 1}{2}}k(-1, t)dt = \frac{a_\pm + a_\mp}{2}B\left(\frac{\alpha+1}{2},\frac{\alpha+1}{2}\right).
\end{align*}
Comparing this to Lemma \ref{lemma: gamma laplasjan}, we get the result.
\end{proof}

\subsection{Servadei--Valdinoci form}\label{ss.SV}
Let $\mass := \Masss$ for $\alpha \in (0,2)$ and $\boldsymbol{a} = \mathcal{A}_\alpha\cdot (1,0,1,1)$. The Servadei--Valdinoci form \eqref{e.SVf}
is
\begin{align}\label{e.SVf2}
    \frac{1}{2} \int_{\R}\int_{\R} \bigl(u(x) - u(y)\bigr)^2 \mass(x,y)  dx  dy = \frac{\mathcal{A}_\alpha}{2} \iint_{U} \frac{\bigl(u(x) - u(y)\bigr)^2}{|x - y|^{\alpha+1}}  dx  dy,
\end{align}
where $U := {\{(x, y) \in \R^2: x > 0 \text{ or } y > 0\}}$.
According to the Introduction, 
\cite[Theorem 1.2.1, Proposition 5.3.7, and Proposition 5.3.9]{BogdanFafulaSztonyk2025} yield a Hardy inequality on $L^2(|x|^{-\alpha}dx)$ for the form. Here we improve this result by proving the criticality of $M_{\mass, 1}$. Moreover,  by Theorem~\ref{theorem: Hardy Inequality}, we provide additional critical $(\mass, 2)$-Hardy weights $M_{\mass, r}$ for all $r > 0$. Below, we present their basic properties. Let
 \begin{equation*}
    Z_{+}(\alpha, r) := C_{+}\left(\mass, r, \frac{\alpha - 1}{2}\right)
    \quad \mbox{ and } \quad
    Z_{-}(\alpha, r) := C_{-}\left(\mass, r, \frac{\alpha - 1}{2}\right).
\end{equation*}
Then 
\begin{align*}
    M_{\mass, r}(x) = |x|^{-\alpha}\begin{cases}
        Z_+(\alpha, r), \quad x \geq 0,
        \\
        Z_{-}(\alpha, r), \quad x < 0.
    \end{cases}
\end{align*}
From Proposition \cref{proposition: constants nu^a}, we get the following result.

\begin{corollary} For $\alpha \in (0, 2)$ and $r \in (0, \infty)$,
    \begin{align*}
    Z_{+}(\alpha, r) = & \mathcal{A}_\alpha B\left(\frac{\alpha+1}{2},\frac{\alpha+1}{2}\right) \left(\frac{1}{\sin(\pi\alpha/2)} - r \right),
    \\
    Z_{-}(\alpha, r) = & \mathcal{A}_\alpha\left ( \frac{1}{\alpha} - \frac{1}{r}B\left(\frac{\alpha+1}{2},\frac{\alpha+1}{2}\right) \right ).
\end{align*}
\end{corollary}
We easily see  that $Z_+, Z_- > 0$ if and only if
\begin{align*}
    r \in \left ( \alpha B \left( \frac{\alpha + 1}{2}, \frac{\alpha + 1}{2} \right), \; \frac{1}{\sin(\pi\alpha/2)}\right) =: U(\alpha). 
\end{align*}
Moreover, $U(\alpha) = \varnothing$ only in the case $\alpha = 1$. By Collorary \cref{corollary: case p = 2 constant}, in the case $\alpha \in (0, 2) \setminus \{1\}$, there exists $r_0 \in U(\alpha)$ such that
\begin{align*}
    Z_+(\alpha, r_0) = Z_-(\alpha, r_0) =: D(\alpha).
\end{align*}
Thus,
\begin{align*}
    M_{\mass, r_0}(x) = D(\alpha)|x|^{-\alpha}, \quad x\in \R,
\end{align*}
is a symmetric critical Hardy weight for the Servadei--Valdinoci form.
The constant $D(\alpha)$ can be calculated explicitly: 
\begin{align}\label{e.sSV}
    D(\alpha) = \frac{1}{2}\mathcal{A}_\alpha b(\alpha)\left ( \frac{1}{ \sin{\left( \frac{\pi\alpha}{2} \right)}} + \frac{1}{\alpha b(\alpha)} - \sqrt{4 
    +\left ( \frac{1}{ \sin{\left( \frac{\pi\alpha}{2} \right)}} - \frac{1}{\alpha b(\alpha)}  \right )^2} \right ),
\end{align}
where $b(\alpha) := B\left(\frac{\alpha+1}{2},\frac{\alpha+1}{2}\right)$.
In the case $\alpha = 1$, $Z_+(1, 1) = Z_-(1, 1) = 0$ and if $r \neq 1$ then $Z_+(1, r)$ and $Z_-(1, r)$ are  nonzero and have opposite signs.
The comparison of the Hardy constant for the quadratic form \eqref{definition: quadratic form} of the fractional Laplacian  and the Servadei--Valdinoci form $S$ in \eqref{e.SVf} is given in Figure \ref{f.cHc}. An informal interpretation guided by probabilistic intuitions of the authors is that the escape rate of mass due to activity of the form \eqref{definition: quadratic form} is higher, especially on the negative half-line, so there is  more room for the Hardy integral than in the case of $S$.
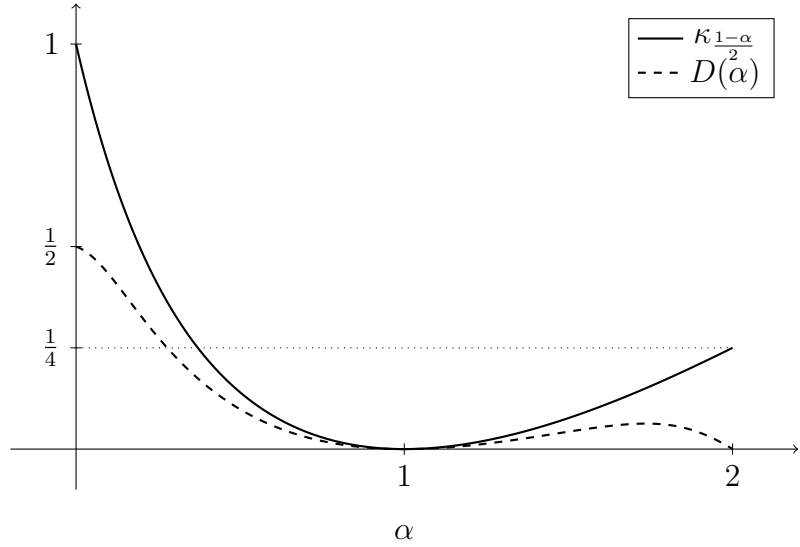
\begin{figure}[H]
\centering
\vspace{2mm} 
\begin{tikzpicture}
\begin{axis}[
    title = {},
    width=12cm,
    height=8cm,
    axis lines=middle,
    axis line style={->},
    tick style={black},
    enlargelimits,
    xlabel={$\alpha$},
    ylabel={},
    legend pos=north east,
    xtick={0,1,2},
    ytick={0,0.25,0.5,1},
    yticklabels={0,$\frac{1}{4}$,$\frac{1}{2}$,1},
    xlabel style={at={(axis description cs:0.5,-0.05)}, anchor=north},
    ylabel style={at={(axis description cs:-0.06,0.5)}, rotate=90, anchor=south}
]
    \addplot[black, thick]
        table[x=alpha, y=kappa, col sep=comma] {kappa.csv};
    \addlegendentry{$\kappa_{\frac{1-\alpha}{2}}$}

    \addplot[black, thick, dashed]
        table[x=alpha, y=D, col sep=comma] {D.csv};
    \addlegendentry{$D(\alpha)$}
    
    \addplot[black, thin, dotted] coordinates {(0,0.25) (2,0.25)};
\end{axis}
\end{tikzpicture}
\caption{Comparison of $\kappa_{\frac{1-\alpha}{2}}$ and $D(\alpha)$.}\label{f.cHc}
\end{figure}

\subsection{The role of the reference measure}\label{ss.uorm}
We offer some comments to further explain the role of the reference measure $|x|^{-\alpha}dx$ in our results. 
We will focus on $p=2$ and derive the Hardy inequality for $L^2(\R,dx)$ as a consequence of the Hardy inequality for $L^2(\R, |x|^{-\alpha}dx)$. We will get the same Hardy weight, but we will also experience a restriction on $\alpha$. For motivation, compare the settings of \eqref{e:hardy-quad} and Corollary \ref{c.pnH}, \ref{corollary: case p = 2 ground state}, or \ref{corollary: case p = 2 critical weight}.
Thus, for \( n \in \mathbb{N} \), we define the truncation function
\[
T_n(x) := \operatorname{sgn}(x) \cdot \left( \left(|x| - \tfrac{1}{n} \right) \vee 0 \wedge n \right),\quad x\in \R.
\]
Of course, $|T_n(x)|\le |x|$, $|T_n(x)-T_n(y)|\le |x-y|$, and $T_n(x)\to x$ as $n\to \infty$. Therefore, by the self-dominated convergence \cite[Lemma 6]{MR4372148}, $\E[T_n(u)]\to \E[u]$ and $\int_\R |T_n(u)|^2\mu\to \int_\R |u|^2$ for every measure $\mu$ on $\R$ and $u:\R\to \R$.
Furthermore, if $u\in L^2(dx)$, then for every $n\in \mathbb N$, the set $\{|u(x)|>1/n\}$ is of finite measure, hence $T_n(u)\in L^2(|x|^{-\alpha}\,dx)$ if $\alpha< 1$.
By the above and Corollary~\ref{c.Hip} with $p=2$, we get for $\beta \in (-2, 2\alpha)$, $s \in (0, \infty)$, $k \in K_\alpha$, 
\begin{align}
\label{equation: hardy inequality V dx}
\begin{split}
        \frac{1}{2}\int_{\R}\int_{\R}|u(x)-u(y)|^2 k(x, y)\,dx\,dy \geq \int_{\R}|u(x)|^2 w_{k, 2, s, \beta}(x)dx, \quad u\in L^2(\R,dx), 
\end{split}
\end{align}
as needed, provided $\alpha< 1$.
The inequality and arguments clearly fail for $\alpha\in (0,1]$, even for $u(x)=e^{-x^2}$, because the function is not square integrable with respect to $|x|^{-\alpha}dx$.

Recall, however, that we do have a nontrivial Hardy inequality for $\alpha\in (1,2)$ and $u\in L^2(\R,|x|^{-\alpha}dx)$, in particular for $u\in C_c(\R_*)$, see Figure~\ref{f.cHc} and Corollary~\ref{c.La2}. This gives further  motivation for using  $|x|^{-\alpha}dx$ as the reference measure. 
\section*{Declarations}

\paragraph{\bf Conflicts of Interests/Competing Interests} 
The authors declare that there are no conflicts of interest or competing interests associated with this work.

\paragraph{\bf Data Availability} 
No data was used in this research.





\end{document}